\documentclass[10pt, reqno]{amsart}
\usepackage{amsmath, amsthm, amscd, amsfonts, amssymb, graphicx, xcolor}
\usepackage[bookmarksnumbered, colorlinks, plainpages]{hyperref}

\usepackage{hyperref}

\hypersetup{colorlinks=true, linkcolor=black, urlcolor=blue, urlbordercolor={0 1 1}, citecolor=black}

\newtheorem{theorem}{Theorem}[section]
\newtheorem{lemma}[theorem]{Lemma}
\newtheorem{proposition}[theorem]{Proposition}
\newtheorem{corollary}[theorem]{Corollary}

\theoremstyle{definition}
\newtheorem{definition}[theorem]{Definition}
\newtheorem{example}[theorem]{Example}

\numberwithin{equation}{section}

\begin{document}
\setcounter{page}{1}

% \color{black}{
% \noindent 
% {\small Annals of Mathematics and Computer Science}\hfill     {\small ISSN: 2789-7206}\\
% {\small Vol 20 (2024) 1-3}\hfill  {\small https://doi.org/10.56947/amcs.v20.223}}

\centerline{}

\centerline{}

%------------------------------------------------------------------------------

%Title of the paper
\title[Hilbert space models and Blaschke frames]{Hilbert space models and Blaschke frames}

%Author names and affiliations
\author[Connor Evans]{Connor Evans}

\address{School of Mathematics, Statistics and Physics, Newcastle University, Newcastle upon
Tyne NE1 7RU, U.K.}
\email{\textcolor[rgb]{0.00,0.00,0.00}{CEvansMathematics@outlook.com}}

\dedicatory{Dedicated to my grandfather, Peter Reid. \\
A man who is nothing short of an eternal sunshine on Leith}

% \date{Received: xxxxxx; Revised: yyyyyy; Accepted: zzzzzz.
% \newline \indent $^{*}$ Corresponding author
% \newline \indent © The Author(s) 2025. This article is licensed under a Creative Commons Attribution-
% \newline \indent NonCommercial-NoDerivatives 4.0
% International License. To view a copy of the licence, visit 
% \newline \indent \url{https://creativecommons.org/licenses/by-nc-nd/4.0/}}

\date{23rd October, 2025.
\newline \indent 2020 {\it{Mathematics Subject Classification}}. 30J10, 42C15, 46B15, 47A60.
\newline \indent {\it{Key words and phrases}}. Frames, Riesz basis, Hilbert space models, Blaschke product, redundancy problem. 
\newline \indent Supported by the the Engineering and Physical Sciences Research Council grant EP/T517914/1}

%Abstract, keywords, math subject classification
 
\begin{abstract}
    \noindent For a finite Blaschke product $B$ and for an isometry $V$ on 
an infinite-dimensional separable complex Hilbert space $\mathcal{H}$ we study a sequence $(b_m)_{m=1}^\infty$ of vectors in $\mathcal{H}$, defined by $b_m = B(V^*)e_m$, where $(e_m)_{m=1}^\infty$ is an orthonormal basis in $\mathcal{H}$. We call $(b_m)_{m=1}^\infty$ a Blaschke frame for $B$ with isometry $V$ on $\mathcal{H}$. We show how instrumental the use of Hilbert space models are in frame theory by completely solving the question of redundancy for a Blaschke frame, that is, what vectors can be removed from the frame $(b_{m})_{m=1}^{\infty}$ such that $(b_{m})_{m\neq k}$ is still a frame? Using the Wold decomposition, we prove that a Blaschke frame can have no redundant vectors (a Riesz basis), finitely many or infinitely many redundant vectors. These unique cases depend on the choice of finite Blaschke product and which isometry one chooses in the construction of a Blaschke frame.

%\newline
%\newline
%\noindent \textit{Keywords.} Analysis, PDEs, machine learning, cybernetics
%\newline
%\noindent \textit{2020 Mathematics Subject Classification.} Primary 46L55; Secondary 44B20
\end{abstract} \maketitle

%\tableofcontents

%------------------------------------------------------------------------------

\section{Introduction}

In this paper, we show that the use of Hilbert space models is invaluable in solving the redundancy problem for a particular family of frames that we call a {\it{Blaschke frame}}. Recall, a sequence $(x_{m})_{m=1}^{\infty}$ of elements in an infinite-dimensional separable complex Hilbert space $\mathcal{H}$ is a {\it{frame}} for $\mathcal{H}$ if there exists constants $A,B>0$ such that 

\begin{equation}\label{framedefeq}
A\|f\|_{\mathcal{H}}^{2} \leq \sum_{m=1}^{\infty}\lvert\langle f,x_{m}\rangle_{\mathcal{H}}\rvert^{2}\leq B\|f\|_{\mathcal{H}}^{2}
\end{equation}

\noindent for all $f\in\mathcal{H}$. Corresponding to a finite Blaschke product $B$ of degree $d$ and an isometry $V$ on $\mathcal{H}$, we define a {\it{Blaschke frame for B with isometry $V$ on $\mathcal{H}$}}, or just a {\it{Blaschke frame on $\mathcal{H}$}}, to be any sequence of vectors $(b_{m})_{m=1}^{\infty}$ in $\mathcal{H}$ expressible in the form

$$b_{m}=B(V^{*})e_{m}$$

\noindent for $m=1,2,\hdots$ and some orthonormal basis $(e_{m})_{m=1}^{\infty}$ of $\mathcal{H}$. Since the spectrum of an isometry $V$ is a subset of the closed unit disc and $B$ is holomorphic on $\overline{\mathbb{D}}$, $B(V^{*}):\mathcal{H}\rightarrow\mathcal{H}$ is defined by means of the Riesz-Dunford or the rational functional calculus for operators \cite[Chapter 1]{AMYOperatorbook}. The operator $B(V^{*})$ is a bounded, surjective operator, Proposition \ref{framepropblaschke}. It is required that $B(V^{*})$ be surjective. This is due to the fact that every Hilbert space frame is of the form $(Te_{m})_{m=1}^{\infty}$ where $T:\mathcal{H}\rightarrow\mathcal{H}$ is a bounded, surjective operator \cite[Theorem 5.5.4]{Christensen}.

Frames were first introduced in the seminal paper of Duffin and Schaeffer in \cite{frameog} to study non-harmonic Fourier series. Since then, countless subjects have benefitted from frame theory, such as advances in quantum mechanics from Balazs and Teofanov \cite{quantum}, multiresolution analysis in Benedetto and Li \cite{multires}, analogue-digital conversion in sampling theory with Eldar \cite{eldar}, and in signal processing by Mallat \cite{mallat}. The {\it{redundancy problem}} for a frame is the reason why frame theory has become increasingly important in recent years, most notably in the development of internet communications and even its use in modelling the brain. For instance, how can we recover an entire data file even if some of the data has become partially corrupted? How can the brain retain memory even though neurons die and regenerate? Loosely, if $(x_{m})_{m=1}^{\infty}$ is a frame for a Hilbert space $\mathcal{H}$, then $\overline{\mathrm{Span}}(x_{m})_{m=1}^{\infty} = \mathcal{H}$, but deleting a vector $x_{k}$ from the frame may still leave us with $\overline{\mathrm{Span}}(x_{m})_{m\neq k}^{\infty} = \mathcal{H}$. When this situation occurs, we call such a vector $x_{k}$ a {\it{redundant vector}}. A vector that is not redundant is termed an {\it{essential vector}}. We state a concrete test to find the redundant and essential vectors of a frame, originally given by Christensen in \cite[Theorem 5.4.7]{Christensen}, which we state for the readers convenience as Theorem \ref{redundancyresult}. We can think of a frame as an over-complete basis with insurance against loss of information. 

It will be our main focus to find the redundant vectors for a Blaschke frame, Section \ref{Blaschkeframes}. For many frames, this is typically not an easy task. Instead, we circumvent this problem by utilising {\it{Hilbert space models}}. Recall that the Schur class on the unit disc, $\mathcal{S}(\mathbb{D})$, is the set of holomorphic functions $f:\mathbb{D}\rightarrow \mathbb{C}$ such that $\|f\|_{\infty} = \sup_{z\in\overline{\mathbb{D}}}|f(z)|\leq 1$, where $\mathbb{D}$ is the unit disc in $\mathbb{C}$. We say that a {\it{model}} for $f\in\mathcal{S}(\mathbb{D})$ is a pair $(\mathcal{M},u)$, where $\mathcal{M}$ is a separable Hilbert space and $u:\mathbb{D}\rightarrow \mathcal{M}$ is a mapping such that

$$1-\overline{f(w)}f(z) = \Big\langle (1-\overline{w}z)u(z),u(w)\Big\rangle_{\mathcal{M}}$$ 

\noindent for all $w,\,z\in\mathbb{D}$, \cite[Definition 2.7]{AMYOperatorbook}. Models are extremely useful tools to study function theoretic properties in terms of Hilbert space geometry. In recent years, an effort has been made to define models on more general domains, for instance on the polydisc by Agler \cite{Agler}, on elliptical domains from Agler, Lykova and Young \cite{AglerLykovaYoung} and, a very special domain, the {\it{symmetrized bidisc}} \cite{symmeterizedbidisc}. We will require a generalisation of this domain, namely the {\it{symmetrized polydisc}}, in Definition \ref{polydiscdef}. 

Finite Blaschke products are functions of the form

\begin{equation}\label{fbpform}
B(z) = \tau\prod_{i=1}^{d}\dfrac{z+\lambda_{i}}{1+\overline{\lambda_{i}}z}
\end{equation}

\noindent for all $z\in\mathbb{D}$, where $\tau\in\overline{\mathbb{D}}$ and $\lambda_{i}\in\mathbb{D}$ for $i=1,\hdots,d$. Since finite Blaschke products are in the Schur class on the unit disc and every function $f\in\mathcal{S}(\mathbb{D})$ has a model \cite[Theorem 2.9]{AMYOperatorbook}, finite Blaschke products must also have a model.

\begin{lemma}{\normalfont{\cite[Lemma 2.47]{AMYOperatorbook}}}\label{blasckhelemma}
    Let $\lambda_{1},\hdots,\lambda_{d}\in\mathbb{D}$ and let $B$ be a finite Blaschke product on $\mathbb{D}$ defined by equation (\ref{fbpform}). For $j=1,\hdots,d$, let 

\begin{equation}\label{TM}
    E_{j}(z):=\dfrac{(1-\lvert\lambda_{j}\rvert^{2})^{\frac{1}{2}}}{1+\overline{\lambda_{j}}z}\prod_{k=1}^{j-1}\dfrac{z+\lambda_{k}}{1+\overline{\lambda_{k}}z}.
\end{equation}

\noindent for all $z\in\mathbb{D}$. The functions $E_{1},\hdots,E_{d}$ are orthonormal in $\mathrm{H}^{2}(\mathbb{D})$, the Hardy space on the disc, and, for all $z,w\in\mathbb{D}$, 

\begin{equation}\label{blaschkeproductequation}
    1-\overline{B(w)}B(z)=\sum_{j=1}^{d}\overline{E_{j}(w)}(1-\overline{w}z)E_{j}(z).
\end{equation}

\end{lemma}

\noindent The functions $(E_{j})_{j=1}^{d}$ defined by equation (\ref{TM}) are known as the {\it{Malmquist-Takenaka}} functions corresponding to the finite Blaschke product $B$, as seen in Malmquist \cite{Malm} and Takenaka \cite{TM}. By equation (\ref{blaschkeproductequation}), the pair $(\mathcal{M},u)$ with $\mathcal{M}=\mathbb{C}^{d}$ and $u(z) = (E_{1}(z),\hdots,E_{d}(z))$, is a model for $B$, that is, for all $z,w\in\mathbb{D}$,

$$1-\overline{B(w)}B(z)=\Big\langle(1-\overline{w}z)u(z),u(w)\Big\rangle_{\mathbb{C}^{d}}.$$

\noindent Note that this formula can be extended to all $z,w\in\overline{\mathbb{D}}$, since $B$, as well as the Malmquist-Takenaka functions corresponding to $B$, are well-defined on $\overline{\mathbb{D}}$. This makes finite Blaschke products the ideal candidate to construct frames with, as they avoid singularities on $\overline{\mathbb{D}}$, allowing us to define bounded surjective operators corresponding to finite Blaschke products by using the Riesz-Dunford functional calculus for a broader range of operators. Additionally, as finite Blaschke products belong to the Schur class, the very powerful methods and techniques of Hilbert space model theory are open to us. As far as we have seen in the literature, there has been no explicit use of Hilbert space models applied to frame theory in this manner, that is, using Hilbert space models of Schur class functions to solve the redundancy problem for frames in Section \ref{Blaschkeframes}.

Exploiting models of finite Blaschke products gives us our two main results, Theorem \ref{normofblaschkevectors} and Lemma \ref{redundancycor}, which we state next. Although Theorem \ref{normofblaschkevectors} is somewhat simple, its importance cannot be understated.

\begin{theorem}\label{normofblaschkevectorsintro}
Let $\mathcal{H}$ be a Hilbert space and let $(E_{j})_{j=1}^{d}$ be the Malmquist-Takenaka functions corresponding to a finite Blaschke product $B$. For every vector $b_{m}$ in a Blaschke frame for $B$ with isometry $V$ on $\mathcal{H}$ for $m=1,2,\hdots$, either $b_{m}\in\mathcal{H}_{1}$ and $\|b_{m}\|_{\mathcal{H}}=1$, where $\mathcal{H}_{1}=\bigcap_{j=0}^{\infty}V^{j}\mathcal{H}$, or there exists a positive integer $s$ such that

$$\|b_{m}\|^{2}_{\mathcal{H}}=1 - \dfrac{1}{\big((s-1)!\big)^{2}}\sum_{j=1}^{d}\Bigg\lvert\dfrac{d^{s-1}E_{j}(0)}{dz^{s-1}}\Bigg\rvert^{2}.$$

\end{theorem}

\noindent From the above result, the norm of each vector $b_{m}$ for $m=1,2,\hdots,$ in a Blaschke frame determines whether it is a redundant or an essential vector. We use this to derive a redundancy criterion for any Blaschke frame, which is the second main result we state here. This is Lemma \ref{redundancycor}. 

\begin{lemma}\label{redundancycorintro}
Let $\mathcal{H}$ be a Hilbert space and let $(E_{j})_{j=1}^{d}$ be the Malmquist-Takenaka functions corresponding to a finite Blaschke product $B$. For every vector $b_{k}$ in a Blaschke frame for $B$ with isometry $V$ on $\mathcal{H}$ for $k=1,2,\hdots$ such that $b_{k}\not\in\mathcal{H}_{1}$, where $\mathcal{H}_{1}=\bigcap_{j=0}^{\infty}V^{j}\mathcal{H}$, there exists a positive integer $s$ such that:

\begin{enumerate}
     \item  If $\dfrac{d^{s-1}E_{j}(0)}{dz^{s-1}} = 0$ for $j=1,\hdots,d$, then $b_{k}$ is essential;
     \vspace{0.2cm}
         \item  If $\dfrac{d^{s-1}E_{j}(0)}{dz^{s-1}} \neq 0$ for at least one $j=1,\hdots,d$, then $b_{k}$ is redundant.
\end{enumerate}
\end{lemma}

The choice of finite Blaschke product and isometry $V$ influences the outcome of the redundancy criterion for a Blaschke frame. This can be summarised in three distinct cases as follows:

\begin{enumerate}
    \item A Blaschke frame for any finite Blaschke product with a unitary operator $U$ on $\mathcal{H}$ is a Riesz basis (Proposition \ref{propunitary});
    \item A Blaschke frame for $z^{d}$ with a pure isometry $V$ on $\mathcal{H}$ has finitely many redundant vectors (Proposition \ref{redundantforzd});
    \item If $B$ is a finite Blaschke product with at least one non-zero root, then every vector $b_{k}$ in a Blaschke frame for $B$ with a pure isometry $V$ on $\mathcal{H}$ such that $b_{k}\in\mathcal{H}_{1}$ is an essential vector, every other vector $b_{k}\not\in\mathcal{H}_{1}$ is a redundant vector (Theorem \ref{everyvecisred}).
\end{enumerate}

\noindent By a pure isometry $V\in\mathcal{B}(\mathcal{H})$, where $\mathcal{B}(\mathcal{H})$ denotes the set of bounded linear operators on a Hilbert space $\mathcal{H}$, we take to mean an isometry $V$ on $\mathcal{H}$ such that there is no non-trivial reducing subspace $\mathcal{K}$ for $V$ such that $V|_{\mathcal{K}}$ is unitary \cite[Section 3]{NagyFoais}. Equivalently, the unitary part in the Wold decomposition, Theorem \ref{Wold}, is trivial. Note that all of the Hilbert spaces considered here will be infinite-dimensional separable and complex, unless explicitly stated.

After giving the gist of the paper with this section, we prove a few results regarding the invertibility of operators and Riesz bases, Lemma \ref{invertibleoperator} and Lemma \ref{dualbasisisbasis}, in Section \ref{constructblaschke}. The operators in this section are not only very natural but also instrumental in the construction of the operator $B(V^{*})$, defined by equation (\ref{blaschkeoperator}) and equation (\ref{blaschkeoperator1}). As a result, we give some basic properties of the operator $B(V^{*})$, such as Lemma \ref{blaschkeoperatorcoisom}. In Section \ref{blaschkeframeframeoperator}, we show that a Blaschke frame is indeed a frame, which is Proposition \ref{framepropblaschke}. In fact, a Blaschke frame is a {\it{Parseval frame}}, that is, it is a frame such that the frame bounds $A,B$ as in equation (\ref{framedefeq}) satisfy $A=B=1$. We also give an explicit frame decomposition for a Blaschke frame, which is Theorem \ref{blaschkeframetheorem}. In Section \ref{Blaschkeframes}, we calculate the norm of each frame element in a Blaschke frame, Theorem \ref{normofblaschkevectors}. Doing so requires us to use the Riesz-Dunford functional calculus with an understanding of unilateral shift operators, Definition \ref{wanderingsubspace} and Lemma \ref{propforonbshifts}. We include these so that the Wold decomposition becomes more accessible to us when calculating the redundancy condition for a Blaschke frame with isometry $V$ on $\mathcal{H}$. We find the redundant vectors of a Blaschke frame by providing an explicit formula by the use of Hilbert space models, which is Lemma \ref{redundancycor}.

\section{A Riesz basis and dual Riesz basis}\label{constructblaschke}

Let $\mathcal{H}$ be a Hilbert space and let $T$ be a bounded invertible operator on $\mathcal{H}$. A {\it{Riesz basis}} for $\mathcal{H}$ is a sequence of the form $(Te_{m})_{m=1}^{\infty}$ where $(e_{m})_{m=1}^{\infty}$ is an orthonormal basis for $\mathcal{H}$ \cite[Definition 3.6.1]{Christensen}. A Riesz basis is indeed a basis and it is also a special type of frame, it is a frame which is {\it{exact}}, that is, every vector in the frame is essential, the removal of any element of the frame means we no longer have a frame. From Christensen \cite[Theorem 7.11]{Christensen}, a sequence $(x_{m})_{m=1}^{\infty}$ is an exact frame if and only if $(x_{m})_{m=1}^{\infty}$ is a Riesz basis. Hence, if $(Te_{m})_{m=1}^{\infty} = (x_{m})_{m=1}^{\infty}$ is a Riesz basis, then there exist constants $A,B>0$ such that equation (\ref{framedefeq}) is satisfied. Moreover, the largest possible value for the constant $A$ and smallest possible value for the constant $B$ are

$$\|T^{-1}\|_{\mathcal{B}(\mathcal{H})}^{-2}~~\text{and}~~\|T\|_{\mathcal{B}(\mathcal{H})}^{2}$$ 

\noindent respectively, see \cite[Proposition 3.6.4]{Christensen}. This is a consequence of the following result.

\begin{lemma}{\normalfont{\cite[Page 30]{CWBPDES}}}\label{GuoWang}
    Suppose that $T$ is a bounded invertible operator on a Hilbert space $\mathcal{H}$ such that $Te_{m}=x_{m}$ for $m=1,2,\hdots$, where $(e_{m})_{m\in\mathbb{N}}$ is an orthonormal basis. Then 

\begin{equation}\label{rieszbasisinq}
    \|T^{-1}\|_{\mathcal{B(H)}}^{-1}\leq \|x_{m}\|_{\mathcal{H}} \leq \|T\|_{\mathcal{B(H)}}~~\text{for all}~m=1,2,\hdots.
\end{equation}

\end{lemma}

% \begin{definition}{\normalfont{\cite[Chapter 1]{AMYOperatorbook}}}\label{rdfunc}
% Fix an operator $T\in\mathcal{B}(\mathcal{H})$. For any open set $U\in \mathbb{C}$ such that $\sigma(T)\subseteq U$, let $\gamma$ be a finite collection of closed rectifiable curves in $U\setminus \sigma(T)$ that winds once around each point in the spectrum of $T$ and no times around each point in the complement of $U$. If $f\in \mathrm{Hol}(U)$, then the Cauchy integral formula, 

% $$f(w) = \dfrac{1}{2\pi i}\int_{\gamma}\dfrac{f(z)}{(z-w)}~dz$$

% \noindent whenever $z\in\sigma(T)$. The {\normalfont{Riesz-Dunford functional calculus}} defines $f(T)$ for $f\in\mathrm{Hol}(U)$ by the substitution $z=T$ in this formula, that is to say, 

% \begin{equation}\label{cauchyintergal}
% f(T) = \dfrac{1}{2\pi i}\int_{\gamma}f(z)(z-T)^{-1}~dz.
% \end{equation}

% \end{definition}

% \noindent Note that this definition makes sense, since the assumption $\gamma\subseteq \mathbb{C}\setminus \sigma(T)$ implies that $w-T$ is invertible for all $w\in \gamma$. Moreover, equation (\ref{cauchyintergal}) depends neither on the choice of contour nor on the choice of $U$, since by Cauchy's theorem, for any vectors $u,v$ in $\mathcal{H}$

% $$\big\langle f(T)u,v\big\rangle_{\mathcal{H}} = \dfrac{1}{2\pi i}\int_{\gamma}f(z)\Big\langle(z-T)^{-1}u,v\Big\rangle_{\mathcal{H}}~dz$$

% \noindent is independent of $\gamma$, \cite[Chapter 1]{AMYOperatorbook}.

Consider the space of all polynomials on the closed unit disc $\overline{\mathbb{D}}$, denoted $\mathbb{P}(\overline{\mathbb{D}})$. Note $\mathbb{P}(\overline{\mathbb{D}})\subset \mathrm{H}^{\infty}(\overline{\mathbb{D}})$ where $\mathrm{H}^{\infty}(\overline{\mathbb{D}})$ (respectively $\mathrm{H}^{\infty}(\mathbb{D})$) is the Hardy space of bounded analytic functions on the closed unit disc (respectively open unit disc) equipped with the supremum norm $\| f \|_{\infty}$. For $p(z) = z^{d}+c_{1}z^{d-1}+\hdots+c_{d}$, we refer to the map $J:\mathbb{P}(\overline{\mathbb{D}})\rightarrow\mathbb{P}(\overline{\mathbb{D}})$ defined by

\begin{equation}\label{exchange}
Jp(z)= z^{d}\overline{p(1/\overline{z})} = 1+\overline{c_{1}}z+\hdots+\overline{c}_{d}z^{d}.
\end{equation}

\noindent as the {\it{exchange map}}, where $Jp(0):=1$. We take this terminology from the exchange matrix, and is meant to imply the ``flipping" of the coefficients between $p(z)$ and $z^{d}\overline{p(1/\overline{z})}$. Define $\overline{q}(z) = 1+\overline{c_{1}}z+\hdots+\overline{c_{d}}z^{d}$ and $q(z) = 1+c_{1}z+\hdots+c_{d}z^{d}$. For $\lambda_{1},\hdots,\lambda_{d}$ in $\mathbb{D}$, the polynomials 

\begin{equation}\label{factorisation}
    p(z) = \prod_{j=1}^{d}(z+\lambda_{j})~~~\text{and}~~~\overline{q}(z)=\prod_{j=1}^{d}(1+\overline{\lambda_{j}}z)
\end{equation}

\noindent for all $z\in\overline{\mathbb{D}}$. 

% \begin{lemma}\label{exchangelemma}
%     Let $J$ be the exchange map on $\mathbb{P}(\overline{\mathbb{D}})$. Then $J^{2} = 1_{\mathbb{P}(\overline{\mathbb{D}})}$ and $\|J\|_{\mathbb{P}(\overline{\mathbb{D}})}=1$.
% \end{lemma}

% \begin{proof}
%     $J^{2} = 1_{\mathbb{P}(\overline{\mathbb{D}})}$ follows immediately from the definition of the exchange map. As $\mathbb{P}(\overline{\mathbb{D}})\subset \mathrm{H}^{\infty}(\overline{\mathbb{D}})$,

%     \begin{align*}
%         \|Jp(z)\|_{\mathrm{H}^{\infty}(\overline{\mathbb{D}})} &= \sup_{z\in\overline{\mathbb{D}}}\Big\lvert z^{d}{\overline{p(1/\overline{z})}}\Big\rvert\\
%         &=\sup_{z\in\overline{\mathbb{D}}}\Big\lvert {p(1/\overline{z})}\Big\rvert\\
%         &=\sup_{z\in\overline{\mathbb{D}}}\Big\lvert \overline{z}^{-d}+c_{1}\overline{z}^{-(d-1)}+\hdots+c_{d}\Big\rvert\\
%         &=\sup_{z\in\overline{\mathbb{D}}}\Big\lvert z^{d}+c_{1}z^{d-1}+\hdots+c_{d}\Big\rvert\\
%         &=\|p(z)\|_{\mathrm{H}^{\infty}(\overline{\mathbb{D}})}.
%     \end{align*}\end{proof}

We now introduce the following fundamental domain.

\begin{definition}\label{polydiscdef}
    Let $\lambda=(\lambda_{1},\hdots,\lambda_{d})\in\mathbb{D}^{d}$ and define $\pi_{d} :\mathbb{D}^{d}\rightarrow \mathbb{C}^{d}$ to be
    
    $$\pi_{d}(\lambda)=(\pi_{d,1}(\lambda),\hdots,\pi_{d,d}(\lambda)),$$ 
    
\noindent where

\begin{equation}\label{symmetrizedpolydisc}
\pi_{d,k}(\lambda)=\sum_{1\leq j_{1}<\hdots<j_{k}\leq d} \lambda_{j_{1}}\hdots \lambda_{j_{k}}.
\end{equation}

\noindent The {\it{symmetrized polydisc}}, denoted as $\mathbb{G}_{d}$, is a proper subset of $ \mathbb{C}^{d}$ such that $\mathbb{G}_{d}=\pi_{d}(\mathbb{D}^{d})$, that is, 

$$\mathbb{G}_{d}=\bigg\{(c_{1},\hdots,c_{d})\in\mathbb{C}^{d}: c_{k}=\pi_{d,k}(\lambda)~\text{for}~\lambda\in\mathbb{D}^{d}~\text{and}~k=1,\hdots,d\bigg\}.$$

\end{definition}

\noindent Throughout, we set $\pi_{d,0}(\lambda)=1$ and we will write $\pi_{d,k}(\lambda)=c_{k}$ for $k=0,1,\hdots,d$, suppressing the dependence on $\lambda\in\mathbb{D}^{d}$. Therefore, the polynomial $p$ of equation (\ref{factorisation}) has all its zeroes inside the unit disc if and only if the coefficients $(c_{1},\hdots,c_{d})$ of $p$ belong to $\mathbb{G}_{d}$, by equation (\ref{symmetrizedpolydisc}). Hence all the zeroes of $q$ lie strictly outside of $\overline{\mathbb{D}}$. 

We now consider invertible operators built from polynomials and the Riesz-Dunford functional calculus. The following theorem, known as {\it{von Neumann's inequality}}, is of great importance to us. Recall that an operator $T$ is a contraction on $\mathcal{H}$ if $\|T\|_{\mathcal{B}(\mathcal{H})}\leq 1$.

\begin{theorem}{\normalfont{\cite[Theorem 1.36]{AMYOperatorbook}}}\label{SchurNeumann}
If $\mathcal{H}$ a Hilbert space, $T\in{\mathcal{B}(\mathcal{H})}$ and $T$ is a contraction, then for any polynomial $p$,

$$ \| p(T) \|_{\mathcal{B}(\mathcal{H})} \leq \max_{z\in\overline{\mathbb{D}}}\lvert p(z)\rvert.$$

\end{theorem} 

Polynomials can be used to construct a very natural, invertible operator, to produce a Riesz basis of a Hilbert space. This operator theoretic viewpoint makes it more amenable to analysis and will be used in later results.

\begin{lemma}\label{invertibleoperator}
Let $\mathcal{H}$ be a Hilbert space and let $T$ be a contraction on $\mathcal{H}$. For $(c_{k})_{k=1}^{d}\in\mathbb{G}_{d}$, the operator
 
 $$q(T)=1_{\mathcal{H}}+c_{1}T+\hdots+c_{d}T^{d}$$ 
 
\noindent is bounded and invertible on $\mathcal{H}$. Consequently, whenever
$q(T)$ is bounded and invertible, the sequence $(q(T)e_{m})_{m=1}^{\infty}$ is a Riesz basis of $\mathcal{H}$ such that

\begin{equation}\label{inequalityriesz}
    \min_{z\in\overline{\mathbb{D}}}|q(z)| \leq \|q(T)e_{m}\|_{\mathcal{H}}\leq \max_{z\in\overline{\mathbb{D}}}|q(z)|
\end{equation}

\noindent for $m=1,2,\hdots$ and any orthonormal basis $(e_{m})_{m=1}^{\infty}$. 

\end{lemma}

\begin{proof}
For $(c_{k})_{k=1}^{d}\in\mathbb{G}_{d}$, there exists a polynomial $q$ defined on $\mathbb{D}$ such that

$$q(z)=\prod_{j=1}^{d}(1+\lambda_{j}z) = \sum_{k=0}^{d}c_{k}z^{k}$$

\noindent for  all $z\in\overline{\mathbb{D}}$ with $c_{0}:=1$ and scalars $\lambda_{1},\hdots,\lambda_{d}\in\mathbb{D}$. By the maximum modulus principle, the maximum of $q$ is attained on the unit circle and $q^{-1}$ is bounded and holomorphic on $\overline{\mathbb{D}}$. Since $\sigma(T)\subseteq\overline{\mathbb{D}}$ for any contraction $T$ on a Hilbert space $\mathcal{H}$, we have that by von Neumann's inequality, Theorem \ref{SchurNeumann}, $q(T)$ is bounded by 

$$ \| q(T) \|_{\mathcal{B}(\mathcal{H})} \leq \max_{z\in\overline{\mathbb{D}}}\lvert q(z)\rvert,$$

\noindent which ensures boundedness. Moreover, the operator $q(T)$ is invertible, as each bracket $(1+\lambda_{j}T)$ for $j=1,\hdots d$ in the factorisation of $q(T)$ is invertible, since $\|\lambda_{j}T\|_{\mathcal{B}(\mathcal{H})}<1$ for $j=1,\hdots,d$. 

 Furthermore, for an orthonormal basis $(e_{m})_{m=1}^{\infty}$ of $\mathcal{H}$, $(q(T)e_{m})_{m=1}^{\infty}$ is a Riesz basis for $\mathcal{H}$ since $q(T)$ is bounded and invertible. Because the sequence $(q(T)e_{m})_{m=1}^{\infty}$ produces a Riesz basis of $\mathcal{H}$ and is generated by the operator $q(T)$, we have inequality (\ref{inequalityriesz}) as an immediate consequence of Lemma \ref{GuoWang}.\end{proof}

Every Riesz basis generates another sequence which itself is also a Riesz basis.

\begin{theorem}{\normalfont{\cite[Theorem 3.6.2]{Christensen}}}\label{dualbasis}
    If $(x_{m})_{m=1}^{\infty}$ is a Riesz basis for $\mathcal{H}$, there exists a unique sequence $(y_{m})_{m=1}^{\infty}$ in $\mathcal{H}$ such that

    \begin{equation}\label{dualbasissetupeqn}
       f=\sum_{m=1}^{\infty}\langle f,y_{m}\rangle_{\mathcal{H}}x_{m} 
    \end{equation}

\noindent for all $f\in\mathcal{H}$. The sequence $(y_{m})_{m=1}^{\infty}$ is also a Riesz basis, where $(x_{m})_{m=1}^{\infty}$ and $(y_{m})_{m=1}^{\infty}$ are biorthogonal. Moreover, the series given by equation (\ref{dualbasissetupeqn}) converges unconditionally for all $f\in\mathcal{H}$.
\end{theorem}

The unique sequence $(y_{m})_{m=1}^{\infty}$ in Theorem \ref{dualbasis} is known as the {\it{dual Riesz basis}} of $(x_{m})$. It is known that if $T\in\mathcal{B(H)}$ is invertible such that $(Te_{m})_{m=1}^{\infty}=(x_{m})_{m=1}^{\infty}$, i.e. the conditions for $(x_{m})_{m=1}^{\infty}$ to be a Riesz basis, then the dual Riesz basis has the expression 

$$(y_{m})_{m=1}^{\infty} = \big((T^{-1})^{*}e_{m}\big)_{m=1}^{\infty},$$

\noindent for $(e_{m})_{m=1}^{\infty}$ the orthonormal basis of $\mathcal{H}$. We write $\overline{q}(z) := 1+\overline{c_{1}}z+\hdots+\overline{c_{d}}z^{d}$ for all $z\in\overline{\mathbb{D}}$ such that the operator 

$$\big(q(T)\big)^{*}=\overline{q}(T^{*})$$

\noindent and 

\begin{equation}\label{inverseqoperator}
q^{-1}(T) = (1/q)(T) = (q(T))^{-1}.
\end{equation}

\begin{lemma}\label{dualbasisisbasis}
Let $\mathcal{H}$ be a Hilbert space and let $T$ be a contraction on $\mathcal{H}$. Define the operator $q(T)$ on $\mathcal{H}$ by

$$q(T)=1_{\mathcal{H}}+c_{1}T+\hdots+c_{d}T^{d}$$

\noindent such that $(c_{k})_{k=1}^{d}\in \mathbb{G}_{d}$. Then $((q^{-1}(T))^{*}e_{m})_{m=1}^{\infty}$ is a Riesz basis for $\mathcal{H}$ and

\begin{equation}\label{inequalityforQdinverseadjoint}
    \dfrac{1}{\max_{z\in\overline{\mathbb{D}}}|q(z)|} \leq \|((q^{-1}(T))^{*}e_{m})\|_{\mathcal{H}}\leq \dfrac{1}{\min_{z\in\overline{\mathbb{D}}}|q(z)|}
\end{equation}

\noindent for $m=1,2,\hdots$ and any orthonormal basis $(e_{m})_{m=1}^{\infty}$.

\end{lemma}

\begin{proof}
    By Theorem \ref{dualbasis}, $((q^{-1}(T))^{*}e_{m})_{m=1}^{\infty}$ is the dual Riesz basis of $(q(T)e_{m})_{m=1}^{\infty}$. By Proposition 3.7.3 of \cite{Christensen}, the dual Riesz basis has bounds

$$\dfrac{1}{\max_{z\in\overline{\mathbb{D}}}|q(z)|}\leq \|\overline{q}^{-1}(T^{*})e_{m}\|_{\mathcal{B}(\mathcal{H})}\leq \dfrac{1}{\min_{z\in\overline{\mathbb{D}}}|q(z)|}.$$\end{proof}

\noindent From here, one could consider an {\it{isometric dilation}} of the contraction $T$ on a Hilbert space $\mathcal{H}$ in Lemma \ref{dualbasisisbasis}. An isometric dilation allows one to lift $T$ to an isometry $V$ on a larger Hilbert space $\mathcal{K}$ containing $\mathcal{H}$ \cite[Section 4]{NagyFoais}. We mention this because the following Proposition requires an isometry to hold. Observe that, for every contraction $T$ on $\mathcal{H}$, there always exists an isometric dilation of $T$ \cite[Theorem 4.1]{NagyFoais}.

The following result points our discussion into the direction of frames.

\begin{proposition}\label{Psurjective}
Let $p(z)=z^{d}+c_{1}z^{d-1}+\hdots+c_{d}$ for all $z\in\overline{\mathbb{D}}$ with coefficients $(c_{k})_{k=1}^{d}\in\mathbb{G}_{d}$. For an isometry $V$ on a Hilbert space $\mathcal{H}$, define the operator $p(V^{*})$ on $\mathcal{H}$ by

$$p(V^{*})=(V^{*})^{d}+c_{1}(V^{*})^{d-1}+\hdots +c_{d}.$$

\noindent Then $p(V^{*})$ is surjective.

\end{proposition}

\begin{proof}
The polynomial $p$ has the factorisation given by equation (\ref{factorisation}). Hence, there exists scalars $\lambda_{1},\hdots,\lambda_{d}\in\mathbb{D}$ such that

\begin{equation}\label{inequalityforinv}
0\leq \dfrac{2\lvert\lambda_{j}\rvert}{1+\lvert\lambda_{j}\rvert^{2}}<1
\end{equation}

\noindent for $j=1,\hdots,d$. As $V$ is an isometry on $\mathcal{H}$, we can write the operator $(V^{*}+\lambda_{j})(V+\overline{\lambda}_{j})$ as

$$
    (V^{*}+\lambda_{j})(V+\overline{\lambda}_{j}) = 1_{\mathcal{H}} + 2\mathrm{Re}(\lambda_{j} V) + \lvert \lambda_{j} \rvert^{2} =(1_{\mathcal{H}}+\lvert\lambda_{j}\rvert^{2})\bigg(1_{\mathcal{H}}+\dfrac{2\mathrm{Re}(\lambda_{j} V)}{1+\lvert\lambda_{j}\rvert^{2}}\bigg).$$

\noindent Since

$$ || 2\mathrm{Re}(\lambda_{j} V) ||_{\mathcal{B}(\mathcal{H})} = || \lambda_{j} V + \overline{\lambda_{j}} V^{*} ||_{\mathcal{B}(\mathcal{H})} \leq 2 \lvert\lambda_{j} \rvert, $$

\noindent we have that by inequality (\ref{inequalityforinv}), 

\begin{equation}\label{inequalityforinv2}
\Bigg\|\dfrac{2\mathrm{Re}(\lambda_{j} V)}{1+\lvert\lambda_{j}\rvert^{2}}\Bigg\|_{\mathcal{B}(\mathcal{H})}\leq \dfrac{2\lvert\lambda_{j}\rvert}{1+\lvert\lambda_{j}\rvert^{2}} < 1,
\end{equation}

\noindent hence the operator $(V^{*}+\lambda_{j})(V+\overline{\lambda}_{j})$ is invertible, as each of the operators $1_{\mathcal{H}}+\lvert\lambda_{j}\rvert^{2}$ and

$$ 1_{\mathcal{H}}+\dfrac{2\mathrm{Re}(\lambda_{j} V)}{1+\lvert\lambda_{j}\rvert^{2}}$$ 

\noindent are invertible for all $\lambda_{j}\in\mathbb{D}$ and $j=1,\hdots,d$.

 A bounded operator is surjective if and only if it has a bounded right inverse. For $\lambda_{j}\in\mathbb{D}$ for $j=1,\hdots,d$, we claim that the operator $V^{*}+\lambda_{j}$ has a bounded right inverse, denoted by $\dagger$, given by

$$(V^{*}+\lambda_{j})^{\dagger}=(V+\overline{\lambda}_{j})\bigg[(V^{*}+\lambda_{j})(V+\overline{\lambda}_{j})\bigg]^{-1}.$$

\noindent Indeed, we have proven invertibility of the square bracket for the operator given above and boundedness can be proven by a Neumann series argument and an application of the triangle inequality. By direct calculation,

$$(V^{*}+\lambda_{j})(V^{*}+\lambda_{j})^{\dagger}
    =(V^{*}+\lambda_{j})(V+\overline{\lambda}_{j})\bigg[(V^{*}+\lambda_{j})(V+\overline{\lambda}_{j})\bigg]^{-1} 
    = 1_{\mathcal{H}}.$$

\noindent Therefore the operator $V^{*}+\lambda_{j}$ is surjective. Furthermore, the operators

$$p(V^{*}) = \prod_{j=1}^{d}(V^{*}+\lambda_{j})$$

\noindent are surjective. This comes as consequence of the composition of surjective operators being surjective.\end{proof}

Given a polynomial $p$ with all its zeroes in the unit disc, we can construct a finite Blaschke product $B$ of degree $d$, equation (\ref{fbpform}), by utilising the exchange map of equation (\ref{exchange}) such that $B(z)=p(z)\overline{q}^{-1}(z)$. Moreover, for a Hilbert space $\mathcal{H}$ and an isometry $V$ on $\mathcal{H}$, we define an operator by the Riesz-Dunford functional calculus, similar to a finite Blaschke product $B$, by

\begin{equation}\label{blaschkeoperator}
B(V^{*}) = p(V^{*})\overline{q}^{-1}(V^{*})
\end{equation}

\noindent where $q(z) = 1+c_{1}z+\hdots+c_{d}z^{d}$ and $p(z)=J\overline{q}(z)$ such that $J$ is the exchange map and

\begin{equation}\label{blaschkeoperator1}
B(V^{*})=\prod_{j=1}^{d}\big(V^{*}+\lambda_{j}\big)\big(1_{\mathcal{H}}+\overline{\lambda}_{j}V^{*}\big)^{-1}.
\end{equation}

\noindent Note $B(V^{*})$ is surjective and has a bounded right inverse $B^{\dagger}(V^{*})$ as follows. Since $q(V^{*})$ is invertible and, by Proposition \ref{Psurjective}, we know that $p(V^{*})$ is surjective, 

\begin{equation}\label{blaschkesurjective}
B^{\dagger}(V^{*}) = q(V^{*})p^{\dagger}(V^{*})
\end{equation}

\noindent such that

$$ B(V^{*})B^{\dagger}(V^{*}) = p(V^{*})\overline{q}^{-1}(V^{*})\overline{q}(V^{*})p^{\dagger}(V^{*}) = 1_{\mathcal{H}}.$$

\begin{lemma}\label{blaschkeoperatorcoisom}
Let $\mathcal{H}$ be a Hilbert space and let $V$ be an isometry on $\mathcal{H}$. For $\lambda_{1}\hdots,\lambda_{d}\in\mathbb{D}$, the operator $B(V^{*})$ on $\mathcal{H}$ defined by

$$B(V^{*}) =\prod_{j=1}^{d}\big(V^{*}+\lambda_{j}\big)\big(1_{\mathcal{H}}+\overline{\lambda_{j}}V^{*}\big)^{-1}$$

\noindent is a co-isometry.

\end{lemma}

\begin{proof}
For $j=1,\hdots,d$, the brackets of the operator $(V^{*}+\lambda_{j}I)(I+\overline{\lambda}_{j}V^{*})^{-1}$ commute, 

\begin{align*}(V^{*}+\lambda_{j}&)(1_{\mathcal{H}}+\overline{\lambda_{j}}V^{*})^{-1} \\
&= (V^{*}+\lambda_{j})\sum_{k=0}^{\infty}(-1)^{k}(\overline{\lambda_{j}}V^{*})^{k} \\
&=\sum_{k=0}^{\infty}(-1)^{k}(\overline{\lambda_{j}}V^{*})^{k}V^{*}+\lambda_{j}\sum_{k=0}^{\infty}(-1)^{k}(\overline{\lambda_{j}}V^{*})^{k}\\
&=(1_{\mathcal{H}}+\overline{\lambda_{j}}V^{*})^{-1}(V^{*}+\lambda_{j}).
\end{align*}

\noindent Moreover, the operator $(V^{*}+\lambda_{j})(1_{\mathcal{H}}+\overline{\lambda}_{j}V^{*})^{-1}$ is a co-isometry for each $\lambda_{j}\in\mathbb{D}$ and $j=1,\hdots,d$. Indeed, by the commutativity of each bracket,

\begin{align*}
    &1_{\mathcal{H}}-(V^{*}+\lambda_{j})(1_{\mathcal{H}}+\overline{\lambda}_{j}V^{*})^{-1}\bigg((V^{*}+\lambda_{j})(1_{\mathcal{H}}+\overline{\lambda}_{j}V^{*})^{-1}\bigg)^{*}\\
    &=1_{\mathcal{H}} - (V^{*}+\lambda_{j})(1_{\mathcal{H}}+\overline{\lambda}_{j}V^{*})^{-1}(1_{\mathcal{H}}+\lambda_{j}V)^{-1}(V+\overline{\lambda_{j}}) \\
    &= 1_{\mathcal{H}} - (1_{\mathcal{H}}+\overline{\lambda}_{j}V^{*})^{-1}(V^{*}+\lambda_{j})(V+\overline{\lambda_{j}})(1_{\mathcal{H}}+\lambda_{j}V)^{-1}\\
    &= (1+\overline{\lambda}_{j}V^{*})^{-1}\bigg((1_{\mathcal{H}}+\overline{\lambda}_{j}V^{*})(1_{\mathcal{H}}+\lambda_{j}V) \\
    &\hspace{3cm}- (V^{*}+\lambda_{j})(V+\overline{\lambda_{j}})\bigg)(1_{\mathcal{H}}+\lambda_{j}V)^{-1}.
\end{align*}

\noindent Expand brackets in the above equation and use the fact that $V$ is an isometry,

\begin{align*}
    &1_{\mathcal{H}}-(V^{*}+\lambda_{j})(1_{\mathcal{H}}+\overline{\lambda}_{j}V^{*})^{-1}\bigg((V^{*}+\lambda_{j})(1_{\mathcal{H}}+\overline{\lambda}_{j}V^{*})^{-1}\bigg)^{*} \\
    &=(1_{\mathcal{H}}+\overline{\lambda}_{j}V^{*})^{-1}\bigg(1_{\mathcal{H}}+2\mathrm{Re}(\lambda_{j}V)+\lvert\lambda_{j}\rvert^{2}\\
    &\hspace{3cm}-\Big(V^{*}V+2\mathrm{Re}(\lambda_{j}V)+\lvert\lambda_{j}\rvert^{2}\Big)\bigg)(1_{\mathcal{H}}+\lambda_{j}V)^{-1}\\
    &=(1_{\mathcal{H}}+\overline{\lambda}_{j}V^{*})^{-1}\bigg(1_{\mathcal{H}}-V^{*}V\bigg)(1_{\mathcal{H}}+\lambda_{j}V)^{-1}.\\
\end{align*}

\noindent Since $V$ is an isometry on $\mathcal{H}$, the brackets in the centre of the equation above become the zero operator on $\mathcal{H}$ and

$$1_{\mathcal{H}}-(V^{*}+\lambda_{j})(1_{\mathcal{H}}+\overline{\lambda}_{j}V^{*})^{-1}\bigg((V^{*}+\lambda_{j})(1_{\mathcal{H}}+\overline{\lambda}_{j}V^{*})^{-1}\bigg)^{*}=0_{\mathcal{H}}.$$

\noindent Hence, for each $\lambda_{j}\in\mathbb{D}$ and $j=1,\hdots,d$,

$$1_{\mathcal{H}}=(V^{*}+\lambda_{j})(1_{\mathcal{H}}+\overline{\lambda}_{j}V^{*})^{-1}\bigg((V^{*}+\lambda_{j})(1_{\mathcal{H}}+\overline{\lambda}_{j}V^{*})^{-1}\bigg)^{*}.$$

Form the operator 

$$B(V^{*}) = \prod_{j=1}^{d}\big(V^{*}+\lambda_{j}\big)\big(1_{\mathcal{H}}+\overline{\lambda}_{j}V^{*}\big)^{-1}.$$

\noindent Since the operators in the product are co-isometries on $\mathcal{H}$ and the composition of co-isometries is also a co-isometry, then $B(V^{*})$ is a co-isometry. Moreover, 

$$\|B(V^{*})\|^{2}_{\mathcal{B}(\mathcal{H})} = \|B(V^{*})B^{*}(V^{*})\|_{\mathcal{B}(\mathcal{H})} = \|1_{\mathcal{H}}\|_{\mathcal{B}(\mathcal{H})} = 1.$$\end{proof}

\section{Blaschke frames and the frame operator}\label{blaschkeframeframeoperator}

Let $\mathcal{H}$ be a Hilbert space and let $V$ be an isometry on $\mathcal{H}$. Define the sequence of vectors $(b_{m})_{m=1}^{\infty}$ in $\mathcal{H}$ such that

\begin{equation}\label{blaschkeframe}
b_{m}=B(V^{*})e_{m},
\end{equation}

\noindent for $m=1,\hdots,d$ and the co-isometry $B(V^{*})$ on $\mathcal{H}$ given by equation (\ref{blaschkeoperator}).

\begin{proposition}{\normalfont{\cite[Corollary 5.3.2]{Christensen}}}\label{closedboundedprop}
     Assume that $(x_{m})_{m=1}^{\infty}$ is a frame for a Hilbert space $\mathcal{H}$ with bounds $A,B$ and that $T$ is a bounded surjective operator on $\mathcal{H}$. Then $(Tx_{m})_{m=1}^{\infty}$ is a frame for $\mathcal{H}$ with frame bounds $A\|T^{\dagger}\|_{\mathcal{B(H)}}^{-2}$, $B\|T\|_{\mathcal{B(H)}}^{2}$ where $T^{\dagger}$ is the right inverse of $T$.
\end{proposition}

Thus, for the sequence $(b_{m})_{m=1}^{\infty}$ defined in equation (\ref{blaschkeframe}),

\begin{equation}\label{boundsforexistenceofframe}
\dfrac{\|p^{\dagger}(V^{*})\|_{\mathcal{B}(\mathcal{H})}^{-1}}{\max_{z\in\overline{\mathbb{D}}}|q(z)|}\|f\|_{\mathcal{H}} \leq \Bigg[\sum_{m=1}^{\infty}\lvert \langle f,b_{m}\rangle_{\mathcal{H}}\rvert^{2}\Bigg]^{1/2}\leq \dfrac{\|p(V^{*})\|_{\mathcal{B}(\mathcal{H})}}{{\min_{z\in\overline{\mathbb{D}}}|q(z)|}}\|f\|_{\mathcal{H}}
\end{equation}

\noindent for all $f\in \mathcal{H}$. The inequality above can be seen if we choose the contraction $T$ in Lemma \ref{dualbasisisbasis} to be an isometry $V$, where $(q(V)e_{m})_{m=1}^{\infty}$ is a frame (in fact it is a Riesz basis), then by finding the dual Riesz basis of $(q(V)e_{m})_{m=1}^{\infty}$ from the operator $\overline{q}^{-1}(V^{*})$ and applying the surjective operator $p(V^{*})$ from Proposition \ref{Psurjective} to the Riesz basis $(\overline{q}^{-1}(V^{*})e_{m})_{m=1}^{\infty}$ to build the sequence from equation (\ref{blaschkeframe}). Proposition \ref{closedboundedprop} gives us that the sequence $(b_{m})_{m=1}^{\infty}$ is a frame for $\mathcal{H}$. That is, we have proved the following result.

\begin{proposition}\label{framepropblaschke}
Let $(e_{m})_{m=1}^{\infty}$ be an orthonormal basis for a Hilbert space $\mathcal{H}$ and let $V$ be an isometry on $\mathcal{H}$. The sequence $(b_{m})_{m=1}^{\infty}=(B(V^{*})e_{m})_{m=1}^{\infty}$ is a frame for $\mathcal{H}$, where $B(V^{*})$ is a co-isometry on $\mathcal{H}$ defined by

$$B(V^{*}) = \prod_{j=1}^{d}\big(V^{*}+\lambda_{j}\big)\big(1_{\mathcal{H}}+\overline{\lambda}_{j}V^{*}\big)^{-1}$$

\noindent and $\lambda_{1},\hdots,\lambda_{d} \in\mathbb{D}$.

\end{proposition}

Let $\mathcal{H}$ be a Hilbert space. Associated with a finite Blaschke product of degree $d$ and an isometry $V$ on $\mathcal{H}$, we define a {\it{Blaschke frame for $B$ with isometry $V$ on $\mathcal{H}$}} to be any sequence of vectors $(b_{m})_{m=1}^{\infty}$ in $\mathcal{H}$ expressible in the form given by equation (\ref{blaschkeframe}). When the frame is made explicit, we instead refer to this sequence of vectors simply as Blaschke frame on $\mathcal{H}$. 

Inequality (\ref{boundsforexistenceofframe}) is somewhat unsatisfactory in that we have not yet found suitable upper bounds for both $\|p(V^{*})\|_{\mathcal{B}(\mathcal{H})}$ and $\|p^{\dagger}(V^{*})\|_{\mathcal{B}(\mathcal{H})}$. Moreover, if we find suitable upper bounds for each, these may not be optimal. We will show that the optimal bounds for a Blaschke frame can be found and that a Blaschke frame is a Parseval frame. Recall, a frame $(x_{m})_{m=1}^{\infty}$ is a Parseval frame if the frame bounds $A,B$ as in equation (\ref{framedefeq}) satisfy $A=B=1$.

Every frame has what is known as the {\it{frame operator}}. This is a bounded, invertible, self-adjoint and positive operator on $\mathcal{H}$ \cite[Lemma 5.1.5]{Christensen}. It is the operator $\mathfrak{F}:\mathcal{H}\rightarrow \mathcal{H}$ such that

\begin{equation}\label{frameoperatordef}
\mathfrak{F}f=\sum_{m=1}^{\infty}\langle f,x_{m}\rangle_{\mathcal{H}}x_{m}
\end{equation}

\noindent for all $f\in\mathcal{H}$. Every frame operator for a frame in a Hilbert space is constructed from the composition of two relatively simpler operators. Recall, a sequence $(x_{m})_{m=1}^{\infty}$ in a Hilbert space $\mathcal{H}$ is called a Bessel sequence if there exists a constant $B>0$ such that 

\begin{equation}\label{besselsequencedef}
\sum_{m=1}^{\infty}\lvert \langle f,x_{m}\rangle_{\mathcal{H}}\rvert^{2}\leq B\|f\|_{\mathcal{H}}
\end{equation}
    
\noindent for all $f\in\mathcal{H}$. Since a frame $(x_{m})_{m=1}^{\infty}$ for $\mathcal{H}$ is a Bessel sequence, the operator $\mathfrak{S}:l^{2}(\mathbb{N})\rightarrow\mathcal{H}$ defined by

$$ \mathfrak{S}(a_{m})_{m=1}^{\infty}=\sum_{m=1}^{\infty}a_{m}x_{m}$$

\noindent for all sequences $(a_{m})$ in $l^{2}(\mathbb{N})$ is bounded \cite[Theorem 3.2.3]{Christensen}. We call $\mathfrak{S}$ is called the {\it{synthesis}} operator or the {\it{pre-frame}} operator. The adjoint operator $\mathfrak{S}^{*}:\mathcal{H}\rightarrow l^{2}(\mathbb{N})$, is given by

$$\mathfrak{S}^{*}f=\{\langle f,x_{m}\rangle_{\mathcal{H}}\}_{m=1}^{\infty}.$$

\noindent By composing $\mathfrak{S}$ and $\mathfrak{S}^{*}$ we obtain the frame operator $\mathfrak{F}=\mathfrak{S}\mathfrak{S}^{*}$ as in equation (\ref{frameoperatordef}). Therefore it will be in our interests to find the synthesis operator for a Blaschke frame $(b_{m})_{m=1}^{\infty}$, as in equation (\ref{blaschkeframe}). We include the following important result, which gives us a way to find the optimal frame bounds for a frame.

\begin{proposition}{\normalfont{\cite[Proposition 5.4.4]{Christensen}}}\label{optimalframebounds}
    Let $(x_{m})_{m=1}^{\infty}$ be a frame for a Hilbert space $\mathcal{H}$ and let $\mathfrak{S}$ and $\mathfrak{F}$ be the synthesis and frame operators for the frame respectively. The optimal frame bounds $A,B$ for a frame $(x_{m})_{m=1}^{\infty}$ are given by 

\begin{align*}
    A&=\|\mathfrak{F}^{-1}\|^{-1}_{\mathcal{B}(\mathcal{H})} = \|\mathfrak{S}^{\dagger}\|^{-2}_{\mathcal{B}(\mathcal{H})},\\
    B&=\|\mathfrak{F}\|_{\mathcal{B}(\mathcal{H})} = \|\mathfrak{S}\|^{2}_{\mathcal{B}(\mathcal{H})},
\end{align*}

\noindent where $\mathfrak{S}^{\dagger}$ is the bounded right inverse of the synthesis operator.
    
\end{proposition}

The synthesis operator for a Blaschke frame is of immense importance. The next result highlights just how important this is, as it will allow us to answer the question of redundancy for a Blaschke frame.

\begin{theorem}{\normalfont{\cite[Theorem 5.5.1]{Christensen}}}\label{existencepreframe}
    A sequence $(x_{m})_{m=1}^{\infty}$ is a frame for a Hilbert space $\mathcal{H}$ if and only if 

$$\mathfrak{S}:(a_{m})_{m=1}^{\infty}\rightarrow \sum_{m=1}^{\infty}a_{m}x_{m}$$

\noindent is a well-defined mapping of $l^{2}(\mathbb{N})$ onto $\mathcal{H}$.
    
\end{theorem}

We can now find the frame operator for a Blaschke frame.

\begin{theorem}\label{blaschkeframetheorem}
Let $\mathcal{H}$ be a Hilbert space and let $B$ be a finite Blaschke product. A Blaschke frame for $B$ with isometry $V$ on $\mathcal{H}$ is a Parseval frame and the frame operator is the identity operator on $\mathcal{H}$. Moreover, 

$$f=\sum_{m=1}^{\infty}\langle f,b_{m}\rangle_{\mathcal{H}}b_{m}$$

\noindent for all $f\in \mathcal{H}$.

\end{theorem}

% By Christensen \cite[Theorem 3.2.3]{Christensen}, to show that a Blaschke frame for $B$ with isometry $V$ is a frame, it is enough to show that $(b_{m})_{m=1}^{\infty}$ is a Bessel sequence, Definition \ref{besselsequencedef}, with a non-zero upper bound. Indeed, by Lemma \ref{existencepreframe} and equation (\ref{boundsforexistenceofframe}), we automatically have that $(b_{m})_{m=1}^{\infty}$ is a Bessel sequence with

% $$\Bigg[\sum_{m=1}^{\infty}\lvert \langle f,b_{m}\rangle_{\mathcal{H}}\rvert^{2}\Bigg]^{1/2}\leq \dfrac{\|p(V^{*})\|_{\mathcal{B}(\mathcal{H})}}{{\min_{z\in\overline{\mathbb{D}}}|q(z)|}}\|f\|_{\mathcal{H}}.$$

% \noindent where the zeroes of $q_d$ lie strictly outside $\overline{\mathbb{D}}$. By von Neumann's inequality, Theorem \ref{SchurNeumann}, since $V$ is an isometry on $\mathcal{H}$,

% $$\|p(V^{*})\|_{\mathcal{B}(\mathcal{H})}\leq {\max_{\substack{z\in\overline{\mathbb{D}}}}}\lvert p(z)\rvert = {\max_{z\in\overline{\mathbb{D}}}}\lvert q(z)\rvert. $$

% \noindent Thus 

% $$\Bigg[\sum_{m=1}^{\infty}\lvert \langle f,b_{m}\rangle_{\mathcal{H}}\rvert^{2}\Bigg]^{1/2}\leq \dfrac{{\max_{z\in\overline{\mathbb{D}}}}\lvert q(z)\rvert}{{\min_{z\in\overline{\mathbb{D}}}|q(z)|}}\|f\|_{\mathcal{H}}$$

% \noindent and we have a Bessel bound for the frame. Therefore the synthesis operator is well-defined.

\begin{proof}
By Proposition \ref{framepropblaschke}, the sequence of vectors $(b_{m})_{m=1}^{\infty}$ given by equation (\ref{blaschkeframe}) is a Blaschke frame for $\mathcal{H}$. Moreover, we claim that the synthesis operator for a Blaschke frame is given by

\begin{equation}\label{synthesisoperator}
    \mathfrak{S}=B(V^{*})L: l^{2}(\mathbb{N})\rightarrow \mathcal{H},
\end{equation}

\noindent where $L: l^{2}(\mathbb{N})\rightarrow \mathcal{H}$ is the Hilbert space isomorphism

$$L(a_{1},a_{2},\hdots) = a_{1}e_{1}+a_{2}e_{2}+\hdots$$

\noindent and $(e_{m})_{m=1}^{\infty}$ is an orthonormal basis of $\mathcal{H}$ for $m=1,2,\hdots$. For any sequence $(a_{m})_{m=1}^{\infty}\in l^{2}(\mathbb{N})$,

$$\mathfrak{S}(a_{m})_{m=1}^{\infty} = B(V^{*})L(a_{m})_{m=1}^{\infty} = B(V^{*})\sum_{m=1}^{\infty}a_{m}e_{m} = \sum_{m=1}^{\infty}a_{m}B(V^{*})e_{m}=\sum_{m=1}^{\infty}a_{m}b_{m}.$$

% \noindent By \cite[Lemma 3.2.1]{Christensen}, the adjoint of the synthesis operator, $\mathfrak{S}^{*}$, is given by

% $$\mathfrak{S}^{*}f = \Big\{\langle f,b_{m}\rangle_{\mathcal{H}}\Big\}_{m=1}^{\infty}.$$

\noindent As $LL^{*}=1_{\mathcal{H}}$, the frame operator 

$$\mathfrak{F}=\mathfrak{S}\mathfrak{S}^{*}=B(V^{*})LL^{*}B^{*}(V^{*})=B(V^{*})1_{\mathcal{H}}B^{*}(V^{*})=B(V^{*})B^{*}(V^{*}).$$

\noindent By Lemma \ref{blaschkeoperatorcoisom}, $B(V^{*})B^{*}(V^{*})=1_{\mathcal{H}}$, thus $\mathfrak{F}=1_{\mathcal{H}}$. By Proposition \ref{optimalframebounds}, the optimal frame bounds are $A=B=1$. Thus, a Blaschke frame is a Parseval frame. Moreover, for all $f\in \mathcal{H}$, since $B(V^{*})$ is a co-isometry,

$$f = B(V^{*})B^{*}(V^{*})f
    = B(V^{*})\sum_{m=1}^{\infty}\langle B^{*}(V^{*})f, e_{m}\rangle_{\mathcal{H}}e_{m}
    = \sum_{m=1}^{\infty}\langle f, b_{m}\rangle_{\mathcal{H}}b_{m}.
$$\end{proof}

% \begin{align*}
%     A &= \|\mathfrak{F}^{-1}\|^{-1}_{\mathcal{B}(\mathcal{H})} = \|\mathfrak{S}^{\dagger}\|^{-2}_{\mathcal{B}(\mathcal{H})},\\
%     B &= \|\mathfrak{F}\|_{\mathcal{B}(\mathcal{H})} = \|\mathfrak{S}\|^{2}_{\mathcal{B}(\mathcal{H})}.
% \end{align*}

% \noindent For the synthesis operator $\mathfrak{S}$ for a Blaschke frame from equation (\ref{synthesisoperator}) given by $\mathfrak{S}=B(V^{*})L$ and the corresponding frame operator $\mathfrak{F}=1_{\mathcal{H}}$, we have 

\section{The redundancy condition for a Blaschke frame with an isometry}\label{Blaschkeframes}

If $\mathcal{H}$ is a Hilbert space and $K$ is a subset of $\mathcal{H}$, we denote by $\mathcal{H}\ominus K$ the orthogonal complement of $K$ in $\mathcal{H}$, that is,

$$\mathcal{H}\ominus K=\Big\{x\in\mathcal{H}: \langle x,y\rangle_{\mathcal{H}}=0~\text{for all}~y\in K\Big\}.$$

\noindent Let $V\in\mathcal{B}(\mathcal{H})$. Then $\mathrm{Ker}(V^{*})=\mathcal{H}\ominus V\mathcal{H}$, \cite[Lemma 1]{severino}.

We recall the definition of a {\it{unilateral shift}}.

\begin{definition}{\normalfont{\cite[Section 1]{NagyFoais}}}\label{wanderingsubspace}
    Let $V$ be an isometry on a Hilbert space $\mathcal{H}$. A subspace $\mathcal{L}$ of $\mathcal{H}$ is called a {\normalfont{wandering space}} for $V$ if $V^{p}\mathcal{L}\perp V^{q}\mathcal{L}$ for every pair of integers $p,q\geq 0$ and $p\neq q$; equivalently, 

    $$V^{j}\mathcal{L}\perp\mathcal{L}~~\text{for}~j=1,\hdots.$$

\noindent An isometry $S$ on $\mathcal{H}$ is called a {\normalfont{unilateral shift}} if there exists a subspace $\mathcal{L}$ of $\mathcal{H}$ which is wandering for $S$ and $\mathcal{H}=\bigoplus_{j=0}^{\infty}S^{j}\mathcal{L}$. In this case $\mathcal{L}$ is called {\normalfont{generating}} for $S$.

\end{definition}

\noindent For a unilateral shift $S$ on a Hilbert space $\mathcal{H}$, $\sigma(S)=\overline{\mathbb{D}}$.

Denote by $\mathbb{N}$ the set of natural numbers $1,2,\hdots$ and $\mathbb{N}_{0}$ as the set of natural numbers including zero. Let $S$ be a unilateral shift on a Hilbert space $\mathcal{H}$ and let $(e_{m})_{m=1}^{\infty}$ be an orthonormal basis for $\mathcal{H}$. For any unilateral shift $S$, there exists an index function $\Lambda:\mathbb{N}\rightarrow\mathbb{N}_{0}$ such that

\begin{equation}\label{kernelofshiftstar}
\mathrm{Ker}(S^{*})=\overline{\mathrm{Span}}\{e_{\Lambda(n)}\}_{n\geq 1}
\end{equation}

\noindent where $e_{0}:=\underline{0}$, that is, $e_{0}$ is defined to be the zero vector in $\mathcal{H}$. A simple example to demonstrate this is when the unilateral shift is the right shift on $\mathcal{H}$ defined by $Se_{m}=e_{m+1}$. We have $\mathrm{Ker}(S^{*})=\overline{\mathrm{Span}}\{e_{1}\}$ and an index function $\Lambda$ that satisfies $\Lambda(1)=1$ and $\Lambda(n)=0$ for $n>1$. Another example involves the multiplication operator $M_{p}$ on a Hilbert space $\mathcal{H}$, where $p$ is a prime number and is defined by

\begin{equation}\label{mpproperty}
M_{p}e_{m}=e_{pm},
\end{equation}

\noindent where $(e_{m})_{m=1}^{\infty}$ is an orthonormal basis for $\mathcal{H}$. Indeed, the proof that $M_{p}$ is a unilateral shift was considered when $\mathcal{H}=L^{2}(0,1)$, the space of square integrable functions on the interval $(0,1)$, by Boulton and Lord in \cite[Lemma 2]{BL}. In this case, $e_{m}(x)=\sqrt{2}\sin(m\pi x)$ for all $x\in(0,1)$ and $m=1,2,\hdots$. Since $\mathrm{Ker}(M_{p}^{*})=\{e_{n}\}_{n\not\equiv_{p} 0}$, the index function $\Lambda$ in this case is the function satisfying $\Lambda(n)=n$ when $n\not\equiv_{p}0$ and zero otherwise. Note that there exists positive integers $n\not\equiv_{p}0$ and $s=1,2,\hdots$ such that $m=np^{s-1}$ for $m=1,2,\hdots$. One can check that the proof found in \cite[Lemma 2]{BL} is not limited to the Hilbert space $L^{2}(0,1)$. 

This acts as motivation for the next result, where we need the understanding of how a unilateral shift acts on orthonormal basis elements.

\begin{lemma}\label{propforonbshifts}
    Let $(e_{m})_{m=1}^{\infty}$ be an orthonormal basis of a Hilbert space $\mathcal{H}$ and let $S$ be a unilateral shift on $\mathcal{H}$ with $\mathrm{Ker}(S^{*})=\overline{\mathrm{Span}}\{e_{\Lambda(n)}\}_{n\geq 1}$ where $\Lambda:\mathbb{N}\rightarrow\mathbb{N}_{0}$ is an index function. There exists a finite sequence of orthonormal basis vectors corresponding to each $e_{m}$, denoted

    $$e_{\Lambda(n)}^{(s-1)},\hdots,e_{\Lambda(n)}^{(0)},$$

    \noindent where $e_{\Lambda(n)}^{(s-1)}=e_{m}$ and $e_{{\Lambda(n)}}^{(0)}=e_{\Lambda(n)}$ for a positive integer $s$ such that

\begin{equation}\label{sequenceforunishift}
S^{*}e_{m}=e_{\Lambda(n)}^{(s-2)},\hdots,~S^{*}e_{\Lambda(n)}^{(1)}=e_{\Lambda(n)}.
\end{equation}

\noindent That is, $(S^{*})^{k}e_{m}=e_{\Lambda(n)}^{(s-1-k)}$ for $k=0,1,\hdots s-1$. 
\end{lemma}

\begin{proof}
    Since $S$ is a unilateral shift on $\mathcal{H}$, we have $\mathrm{Ker}(S^{*})$ is a wandering subspace for $S$, that is, $\mathcal{H}=\oplus_{j=0}^{\infty} S^{j}\mathcal{L}$ where $\mathcal{L}=\mathrm{Ker}(S^{*})$. Note that the zero vector $e_{0}=\underline{0}\in\mathrm{Ker}(S^{*})$ and $S^{*}e_{0}=e_{0}$. Instead, choose any $n_{1}$ and $n_{2}$ such that $e_{\Lambda(n_{1})}$ and $e_{\Lambda(n_{2})}$ belong to $\mathrm{Ker}(S^{*})$ and are not the zero vector. Then

$$\Big\langle S^{k_{1}}e_{\Lambda(n_{1})},S^{k_{2}}e_{\Lambda(n_{2})}\Big\rangle_{\mathcal{H}}=0$$

\noindent for $k_{1},k_{2}=0,1,\hdots$ such that $k_{1}\neq k_{2}$ when $n_{1}=n_{2}$. Moreover, as $S$ is an isometry, $\|S^{k}e_{\Lambda(n)}\|_{\mathcal{H}}=1$ for a positive integer $n$ and $k=0,1,\hdots$. Therefore, for each $k=0,1,\hdots$ each vector in the sequence $(S^{k}e_{\Lambda(n)})_{n\geq 1}$ is orthonormal in $\mathcal{H}$. Moreover, when $k\neq 0$, the subspace spanned by $(S^{k}e_{\Lambda(n)})_{n\geq 1}$ is orthogonal to $\mathrm{Ker}(S^{*})$. Hence, $\cup_{k=0}^{\infty} (S^{k}e_{\Lambda(n)})_{n\geq 1} =(e_{m})_{m=1}^{\infty}$. 

For some $s=1,2,\hdots$ and a positive integer $n$ such that $\Lambda(n)\neq 0$, let $e_{m}=S^{s-1}e_{\Lambda(n)}$ for an $m=1,2,\hdots$. We can define a finite sequence of orthonormal basis vectors corresponding to $e_{m}$ such that

$$Se_{\Lambda(n)}=e_{\Lambda(n)}^{(1)},~ Se^{(1)}_{\Lambda(n)}=e_{\Lambda(n)}^{(2)},\hdots,Se_{\Lambda(n)}^{(s-2)}=e_{m}.$$

\noindent Equivalently, we have that equation (\ref{sequenceforunishift}) holds, where $e_{\Lambda(n)}^{(s-1)}=e_{m}$ and $e_{\Lambda(n)}^{(0)}=e_{\Lambda(n)}$.\end{proof}

For a unilateral shift $S$ on a Hilbert space $\mathcal{H}$, we know that $\mathrm{Ker}(S^{*})$ is a wandering subspace for $S$. For a general isometry $V$, we cannot make this claim. Instead, observe that the kernel of an operator $V$ in $\mathcal{H}$ is always a closed subspace of $\mathcal{H}$. Therefore $\bigoplus_{j=0}^{\infty}V^{j}\mathcal{H}$ is a closed subspace in $\mathcal{H}$. This is the basis for the next theorem, the {\it{Wold decomposition}} of an isometry $V$ on a Hilbert space, \cite[Theorem 1.1]{NagyFoais} and \cite[Section 1.3]{RR}.

\begin{theorem}{\normalfont{\cite[Theorem 2]{severino}}}\label{Wold}
    Let $\mathcal{H}$ be a Hilbert space and let $V$ be an isometry on $\mathcal{H}$. Then $\mathcal{H}$ decomposes uniquely into an orthogonal sum $\mathcal{H}=\mathcal{H}_{1}\oplus \mathcal{H}_{2}$ such that

    $$V\mathcal{H}_{1}=\mathcal{H}_{1},~~V\mathcal{H}_{2}\subset \mathcal{H}_{2}.$$

\noindent Moreover, the restriction of $V$ on $\mathcal{H}_{1}$, $V|_{\mathcal{H}_{1}}$, is unitary. The restriction of $V$ on $\mathcal{H}_{2}$, $V|_{\mathcal{H}_{2}}$, is a unilateral shift. In particular, for 

$$\mathcal{L}=\mathcal{H}\ominus V\mathcal{H},$$

\noindent then

$$\mathcal{H}_{1}=\bigcap_{j=0}^{\infty}V^{j}\mathcal{H},~~\mathcal{H}_{2}=\bigoplus_{j=0}^{\infty}V^{j}\mathcal{L}.$$
    
\end{theorem}

\noindent Note that for the Wold decomposition, since $\mathrm{Ker}(V^{*})=\mathcal{H}\ominus V\mathcal{H}$, we can rewrite the condition $\mathcal{L}=\mathcal{H}\ominus V\mathcal{H}$ as $\mathcal{L}=\mathrm{Ker}(V^{*})$.

Let $B$ be a finite Blaschke product of degree $d$. Since $B\in\mathrm{Hol}(\overline{\mathbb{D}})$, we have that $B$ can be extended analytically to an open set $U$ including $\overline{\mathbb{D}}$ as a subset, such that no poles of $B$ lie inside $U$. Hence $B\in\mathrm{Hol}(U)$. Such a choice of open set $U$ always exists, as the poles of $B$ lie strictly outside $\overline{\mathbb{D}}$. Thus, a contour of small enough radius can always be chosen to exclude all the poles of $B$. We require such a choice of $U$ as $\sigma(V^{*})\subseteq\overline{\mathbb{D}}$.

Let $\mathcal{H}$ be a Hilbert space and let $V$ be an isometry on $\mathcal{H}$. Let $U$ be an open set in $\mathbb{C}$ such that $\overline{\mathbb{D}}\subset U$ where no poles of the function $B$ lie inside $U$ and let $\gamma$ be a closed rectifiable contour in $U \setminus \overline{\mathbb{D}}$ which winds once around each point in the spectrum of $\sigma(V^{*})\subseteq\overline{\mathbb{D}}$ and no times around each point in the complement of $U$. The Riesz-Dunford functional calculus defines $B(V^{*})$ for $B\in\mathrm{Hol}(U)$ by

\begin{equation}\label{rieszdunufordblaschkeshift}
    B(V^{*})=\dfrac{1}{2\pi i}\int_{\gamma}B(z)(z-V^{*})^{-1}~dz.
\end{equation}

In light of this set up, we have the following result regarding the norm of the vectors $b_{m}$ of a Blaschke frame for $m=1,2,\hdots$. The use of Hilbert space models, not only simplifies the calculation, but is also necessary to solve the redundancy question for a Blaschke frame for $B$ with an isometry $V$ on a Hilbert space $\mathcal{H}$. Recall from equation (\ref{TM}) the Malmquist-Takenaka functions corresponding to a finite Blaschke product $B$, equation (\ref{fbpform}), defined by

$$E_{j}(z):=\dfrac{(1-\lvert\lambda_{j}\rvert^{2})^{\frac{1}{2}}}{1+\overline{\lambda_{j}}z}\prod_{k=1}^{j-1}\dfrac{z+\lambda_{k}}{1+\overline{\lambda_{k}}z},$$

\noindent for all $z\in\mathbb{D}$ and $j=1,\hdots,d$.

\begin{theorem}\label{normofblaschkevectors}
Let $\mathcal{H}$ be a Hilbert space and let $(E_{j})_{j=1}^{d}$ be the Malmquist-Takenaka functions corresponding to a finite Blaschke product $B$. For every vector $b_{m}$ in a Blaschke frame for $B$ with isometry $V$ on $\mathcal{H}$ for $m=1,2,\hdots$, either $b_{m}\in\mathcal{H}_{1}$ and $\|b_{m}\|_{\mathcal{H}}=1$, where $\mathcal{H}_{1}=\bigcap_{j=0}^{\infty}V^{j}\mathcal{H}$, or there exists a positive integer $s$ such that

$$\|b_{m}\|^{2}_{\mathcal{H}}=1 - \dfrac{1}{\big((s-1)!\big)^{2}}\sum_{j=1}^{d}\Bigg\lvert\dfrac{d^{s-1}E_{j}(0)}{dz^{s-1}}\Bigg\rvert^{2}.$$

\end{theorem}

\begin{proof}

By Lemma \ref{blasckhelemma}, there exists a model $(\mathbb{C}^{d},u)$ for any finite Blaschke product $B$, that is, there exists an analytic function $u:\mathbb{D}\rightarrow \mathbb{C}^{d}$ defined by $u(z) = (E_{1}(z),\hdots,E(z))$ such that

$$1-\overline{B(w)}B(z)=\sum_{j=1}^{d}\overline{E_{j}(w)}(1-\overline{w}z)E_{j}(z)$$

\noindent for all $z,w\in\mathbb{D}$. Note that the model can be extended analytically to $\overline{\mathbb{D}}$, since the poles of $B$ and the Malmquist-Takenaka functions are strictly outside $\overline{\mathbb{D}}$. The operator $B(V^{*})$ is defined by the Riesz-Dunford functional calculus by equation (\ref{rieszdunufordblaschkeshift}). A similar argument defines $E_{j}$ for each $j=1,\hdots,d$. Therefore, 

$$1_{\mathcal{H}}-B^{*}(V^{*})B(V^{*})=\sum_{j=1}^{d}E^{*}_{j}(V^{*})(1_{\mathcal{H}}-VV^{*})E_{j}(V^{*}).$$

By the Wold decomposition, $\mathcal{H}=\mathcal{H}_{1}\oplus\mathcal{H}_{2}$ such that $V|_{\mathcal{H}_{1}}$ is unitary on $\mathcal{H}_{1}$ and $V|_{\mathcal{H}_{2}}$ is a unilateral shift on $\mathcal{H}_{2}$. Let $V|_{\mathcal{H}_{1}}=U$ and $V|_{\mathcal{H}_{2}}=S$. Then $V$ has the matrix representation on $\mathcal{H}=\mathcal{H}_{1}\oplus\mathcal{H}_{2}$ given by

$$V=\begin{bmatrix}
U & 0\\
0 & S
\end{bmatrix}:\mathcal{H}_{1}\oplus\mathcal{H}_{2}\rightarrow\mathcal{H}_{1}\oplus\mathcal{H}_{2}.$$

\noindent As $U$ is unitary on $\mathcal{H}_{1}$ and $S$ is a unilateral shift on $\mathcal{H}_{2}$ with wandering subspace $\mathrm{Ker}(V^{*})$, Theorem \ref{Wold}, we have 

$$1_{\mathcal{H}}-VV^{*} = \begin{bmatrix}
1-UU^{*} & 0\\
0& 1-SS^{*}
\end{bmatrix} = \begin{bmatrix}
0_{\mathcal{H}_{1}} & 0\\
0& P_{\mathrm{Ker}(V^{*})}
\end{bmatrix},$$

\noindent  where $P_{\mathrm{Ker}(V^{*})}:\mathcal{H}_{2}\rightarrow\mathcal{H}_{2}$ is the orthogonal projection of $\mathcal{H}_{2}$ onto $\mathrm{Ker}(V^{*})$. Hence,

\begin{equation}\label{Blaschkeoperatorvaluedmodel1}
1_{\mathcal{H}}-B^{*}(V^{*})B(V^{*})=\sum_{j=1}^{d}E^{*}_{j}(V^{*})\begin{bmatrix}
0_{\mathcal{H}_{1}} & 0\\
0& P_{\mathrm{Ker}(V^{*})}
\end{bmatrix}E_{j}(V^{*}).
\end{equation}

\noindent Moreover,

$$\|b_{m}\|_{\mathcal{H}}^{2}=\Big\langle b_{m}, b_{m}\Big\rangle_{\mathcal{H}}
    =\Big\langle B(V^{*})e_{m}, B(V^{*})e_{m}\Big\rangle_{\mathcal{H}}
    =\Big\langle  B^{*}(V^{*})B(V^{*})e_{m},e_{m}\Big\rangle_{\mathcal{H}}.$$

\noindent By equation (\ref{Blaschkeoperatorvaluedmodel1}), 

$$\|b_{m}\|_{\mathcal{H}}^{2}= 1 - \sum_{j=1}^{d}\bigg\langle E^{*}_{j}(V^{*})\begin{bmatrix}
0_{\mathcal{H}_{1}} & 0\\
0& P_{\mathrm{Ker}(V^{*})}
\end{bmatrix}E_{j}(V^{*})e_{m},e_{m}\bigg\rangle_{\mathcal{H}}.$$

\noindent By the Riesz-Dunford functional calculus, there exist open sets $U_{1}$ and $U_{2}$ in $\mathbb{C}$, excluding the poles of the function $B$, for which $\overline{\mathbb{D}}\subset U_{1}$, $\overline{\mathbb{D}}\subset U_{2}$ and closed, rectifiable contours $\gamma_{1},\gamma_{2}$ in $U_{1}\setminus\overline{\mathbb{D}}$ and $U_{2}\setminus\overline{\mathbb{D}}$ respectively such that, for each $j=1,\hdots,d$,

\begin{align*}
&\bigg\langle E^{*}_{j}(V^{*})\begin{bmatrix}
0_{\mathcal{H}_{1}} & 0\\
0& P_{\mathrm{Ker}(V^{*})}
\end{bmatrix}E_{j}(V^{*})e_{m},e_{m}\bigg\rangle_{\mathcal{H}}\\
&=\Big(\dfrac{1}{2\pi i}\Big)\overline{\Big(\dfrac{1}{2\pi i}\Big)}\int_{\gamma_{1}}\int_{\gamma_{2}}\\
&\hspace{3cm}\bigg\langle \begin{bmatrix}
0_{\mathcal{H}_{1}} & 0\\
0& P_{\mathrm{Ker}(V^{*})}
\end{bmatrix}E_{j}(z)(z-V^{*})^{-1}e_{m},E_{j}(w)(w-V^{*})^{-1}e_{m}\bigg\rangle_{\mathcal{H}}~~dz~dw.\\
&=\Big(\dfrac{1}{2\pi i}\Big)\overline{\Big(\dfrac{1}{2\pi i}\Big)}\int_{\gamma_{1}}\int_{\gamma_{2}}\bigg(\dfrac{E_{j}(z)}{z}\bigg)\bigg(\overline{\dfrac{E_{j}(w)}{w}}\bigg)\\
&\hspace{4cm}\cdot\bigg\langle \begin{bmatrix}
0_{\mathcal{H}_{1}} & 0\\
0& P_{\mathrm{Ker}(V^{*})}
\end{bmatrix}\sum_{k=0}^{\infty}z^{-k}(V^{*})^{k}e_{m},
\sum_{l=0}^{\infty}w^{-l}(V^{*})^{l}e_{m}\bigg\rangle_{\mathcal{H}}~~dz~dw.
\end{align*}

Let $e_{m}=f_{m}+g_{m}$, where $f_{m}=P_{\mathcal{H}_{1}}e_{m}$ and $g_{m}=P_{\mathcal{H}_{2}}e_{m}$ such that $P_{\mathcal{H}_{j}}:\mathcal{H}\rightarrow \mathcal{H}$ is the orthogonal projection from $\mathcal{H}$ onto $\mathcal{H}_{j}$ for $j=1,2$. Then

$$\begin{bmatrix}
0_{\mathcal{H}_{1}} & 0\\
0& P_{\mathrm{Ker}(V^{*})}
\end{bmatrix}\sum_{k=0}^{\infty}z^{-k}(V^{*})^{k}e_{m} = \sum_{k=0}^{\infty}z^{-k}
 P_{\mathrm{Ker}(V^{*})}(S^{*})^{k}g_{m}.$$

\noindent Note that if $e_{m}=f_{m}$, that is, when $e_{m}\in\mathcal{H}_{1}$, then the above equation equals zero. In this case, $g_{m}=0$ and the statement of the theorem holds with $\|b_{m}\|=1$ such that $b_{m}\in\mathcal{H}_{1}$. 

Alternatively, suppose $e_{m}\neq f_{m}$. By the Wold decomposition, $\mathcal{L}=\mathrm{Ker}(V^{*})$ is a wandering subspace for the unilateral shift $S$. By equation (\ref{kernelofshiftstar}), there exists an index function $\Lambda:\mathbb{N}\rightarrow \mathbb{N}_{0}$ such that $\mathcal{L}=\overline{\mathrm{Span}}\big\{e_{\Lambda(n)}\big\}_{n\geq 1}$ and $\mathcal{H}_{2}=\oplus_{j=0}^{\infty}V^{j}\mathcal{L}$. The projection of $e_{m}$ onto $\mathcal{H}_{2}$ gives us 

$$g_{m}=\sum_{l=0}^{\infty}\sum_{n\geq 1} \big\langle e_{m}, S^{l}e_{\Lambda(n)}\big\rangle_{\mathcal{H}_{2}}S^{l}e_{\Lambda(n)}$$

\noindent such that $e_{\Lambda(n)}\neq \underline{0}$, that is, $\Lambda(n)\neq 0$. Thus, for $k=0,1,\hdots$,

\begin{align*} P_{\mathrm{Ker}(V^{*})}(S^{*})^{k}g_{m}&=  P_{\mathrm{Ker}(V^{*})}(S^{*})^{k}\sum_{l=0}^{\infty}\sum_{n\geq 1} \big\langle e_{m}, S^{l}e_{\Lambda(n)}\big\rangle_{\mathcal{H}_{2}}S^{l}e_{\Lambda(n)}\\
&= \sum_{l=0}^{\infty}\sum_{n\geq 1} \big\langle e_{m}, S^{l}e_{\Lambda(n)}\big\rangle_{\mathcal{H}_{2}}P_{\mathrm{Ker}(V^{*})}S^{l-k}e_{\Lambda(n)}.
\end{align*}

\noindent For $l=0,1,\hdots$, we have $P_{\mathrm{Ker}(V^{*})}S^{l-k}e_{\Lambda(n)}$ is non-zero if and only if $l=k$. Hence,

$$P_{\mathrm{Ker}(V^{*})}(S^{*})^{k}g_{m}= \sum_{n\geq 1} \big\langle e_{m}, S^{k}e_{\Lambda(n)}\big\rangle_{\mathcal{H}_{2}}e_{\Lambda(n)}.$$

\noindent  Therefore,

\begin{equation}\label{disherem8}
\begin{bmatrix}
0_{\mathcal{H}_{1}} & 0\\
0& P_{\mathrm{Ker}(V^{*})}
\end{bmatrix}\sum_{k=0}^{\infty}z^{-k}(V^{*})^{k}e_{m}=\sum_{k=0}^{\infty}  \sum_{n\geq 1} z^{-k}\big\langle e_{m}, S^{k}e_{\Lambda(n)}\big\rangle_{\mathcal{H}_{2}}e_{\Lambda(n)}.
\end{equation}

By Lemma \ref{propforonbshifts}, there exists a finite sequence of orthonormal basis vectors $(e_{\Lambda(n)}^{(k)})_{k=0}^{s-1}$ corresponding to $e_{m}$ such that $(S^{*})^{k}e_{m}=e_{\Lambda(n)}^{(s-1-k)}$ for $k=0,1,\hdots s-1$, where $e_{m}=e_{\Lambda(n)}^{(s-1)}$ and $e_{\Lambda(n)}=e_{\Lambda(n)}^{(0)}$. Therefore, the expression $\big\langle e_{m}, S^{k}e_{\Lambda(n)}\big\rangle_{\mathcal{H}_{2}}$ from equation (\ref{disherem8}) is non-zero if and only if $k=s-1$ and $n$ is the positive integer such that $S^{s-1}e_{\Lambda(n)}=e_{m}$. Hence, for $e_{m}=e_{\Lambda(n)}^{(s-1)}$,

$$\begin{bmatrix}
0_{\mathcal{H}_{1}} & 0\\
0& P_{\mathrm{Ker}(V^{*})}
\end{bmatrix}\sum_{k=0}^{\infty}z^{-k}(V^{*})^{k}e_{m}=z^{-(s-1)}e_{\Lambda(n)},$$

\noindent where $e_{\Lambda(n)}\in\mathrm{Ker}(S^{*})$. Thus, for each $j=1,\hdots,d$,

\begin{align}
&\bigg\langle E^{*}_{j}(V^{*})\begin{bmatrix}
0_{\mathcal{H}_{1}} & 0\\
0& P_{\mathrm{Ker}(V^{*})}
\end{bmatrix}E_{j}(V^{*})e_{m}.e_{m}\bigg\rangle_{\mathcal{H}}\nonumber\\
&=\Big(\dfrac{1}{2\pi i}\Big)\overline{\Big(\dfrac{1}{2\pi i}\Big)}\int_{\gamma_{1}}\int_{\gamma_{2}}\bigg(\dfrac{E_{j}(z)}{z}\bigg)\bigg(\overline{\dfrac{E_{j}(w)}{w}}\bigg)\nonumber\\
&\hspace{4cm}\cdot\bigg\langle z^{-(s-1)}\begin{bmatrix}
0\\
e_{\Lambda(n)}\end{bmatrix},
\sum_{l=0}^{\infty}w^{-l}(V^{*})^{l}e_{m}\bigg\rangle_{\mathcal{H}}~~dz~dw.\label{hello}
\end{align}

For $g_{m}=\sum_{r=0}^{\infty}\sum_{n\geq 1} \big\langle e_{m}, S^{r}e_{\Lambda(n)}\big\rangle_{\mathcal{H}}S^{r}e_{\Lambda(n)}$,

\begin{align*}\sum_{l=0}^{\infty}w^{-l}(V^{*})^{l}e_{m}=\sum_{l=0}^{\infty}w^{-l}\begin{bmatrix}
    (U^{*})^{l}f_{m}\\
    (S^{*})^{l}g_{m}
\end{bmatrix}=\sum_{l=0}^{\infty}w^{-l}\begin{bmatrix}
    (U^{*})^{l}f_{m}\\
    \sum_{r=0}^{\infty}\sum_{n\geq 1}\langle e_{m},S^{r}e_{\Lambda(n)}\rangle_{\mathcal{H}_{2}}S^{r-l}e_{\Lambda(n)}
\end{bmatrix}.
\end{align*}

From equation (\ref{hello}), 

\begin{align}
&\bigg\langle z^{-(s-1)}\begin{bmatrix}
0\\
e_{\Lambda(n)}\end{bmatrix},
\sum_{l=0}^{\infty}w^{-l}(V^{*})^{l}e_{m}\bigg\rangle_{\mathcal{H}} \nonumber\\
&= \bigg\langle z^{-(s-1)}\begin{bmatrix}
0\\
e_{\Lambda(n)}\end{bmatrix}, \sum_{l=0}^{\infty}w^{-l}\begin{bmatrix}
    (U^{*})^{l}f_{m}\\
    \sum_{r=0}^{\infty}\sum_{n\geq 1}\langle e_{m},S^{r}e_{\Lambda(n)}\rangle_{\mathcal{H}_{2}}S^{r-l}e_{\Lambda(n)}
\end{bmatrix}\bigg\rangle_{\mathcal{H}}\nonumber\\
&=\bigg\langle z^{-(s-1)}
e_{\Lambda(n)}, \sum_{l=0}^{\infty}w^{-l}
    \sum_{r=0}^{\infty}\sum_{n\geq 1}\langle e_{m},S^{r}e_{\Lambda(n)}\rangle_{\mathcal{H}_{2}}S^{r-l}e_{\Lambda(n)}
\bigg\rangle_{\mathcal{H}_{2}}.\label{hiz}
\end{align}

\noindent The inner product $\langle e_{m},S^{r}e_{\Lambda(n)}\rangle_{\mathcal{H}_{2}}$ is non-zero if and only if $r=s-1$ and $n$ is the positive integer such that $e_{m}=S^{s-1}e_{\Lambda(n)}$. Furthermore, the inner product from equation (\ref{hiz}) is non-zero if and only if $l=s-1$. Therefore $r=l=s-1$ and

$$\bigg\langle z^{-(s-1)}\begin{bmatrix}
0\\
e_{\Lambda(n)}\end{bmatrix},
\sum_{l=0}^{\infty}w^{-l}(V^{*})^{l}e_{m}\bigg\rangle_{\mathcal{H}} = \bigg\langle z^{-(s-1)}
e_{\Lambda(n)},
w^{-(s-1)}e_{\Lambda(n)}\bigg\rangle_{\mathcal{H}_{2}}=(z\overline{w})^{-(s-1)}.$$

\noindent Thus,

\begin{align*}
\bigg\langle E^{*}_{j}(V^{*})\begin{bmatrix}
0_{\mathcal{H}_{1}} & 0\\
0& P_{\mathrm{Ker}(V^{*})}
\end{bmatrix}&E_{j}(V^{*})e_{m},e_{m}\bigg\rangle_{\mathcal{H}}\\
&=\Big(\dfrac{1}{2\pi i}\Big)\overline{\Big(\dfrac{1}{2\pi i}\Big)}\int_{\gamma_{1}}\int_{\gamma_{2}}\bigg(\dfrac{E_{j}(z)}{z}\bigg)\bigg(\overline{\dfrac{E_{j}(w)}{w}}\bigg)(z\overline{w})^{-(s-1)}~~dz~dw\\
&= \bigg(\dfrac{1}{2\pi i}\int_{\gamma_{1}} \dfrac{E_{j}(z)}{z^{s}}~dz\bigg)\overline{\bigg(\dfrac{1}{2\pi i}\int_{\gamma_{2}}\dfrac{E_{j}(w)}{w^{s}}~dw\bigg)}\\
&=\Bigg(\dfrac{1}{(s-1)!}\dfrac{d^{s-1}E_{j}(0)}{dz^{s-1}}\Bigg)\overline{\Bigg(\dfrac{1}{(s-1)!}\dfrac{d^{s-1}E_{j}(0)}{dw^{s-1}}\Bigg)}\\
&=\dfrac{1}{\big((s-1)!\big)^{2}}\Bigg\lvert\dfrac{d^{s-1}E_{j}(0)}{dz^{s-1}}\Bigg\rvert^{2}.
\end{align*}

\noindent Hence, for each $m=1,2,\hdots$ such that $b_{m}\not\in\mathcal{H}_{1}$,

$$\|b_{m}\|_{\mathcal{H}}^{2} = 1 - \sum_{j=1}^{d}\bigg\langle P_{\mathrm{Ker}(V^{*})}E_{j}(V^{*})e_{m},E_{j}(V^{*})e_{m}\bigg\rangle_{\mathcal{H}}
    = 1 - \dfrac{1}{\big((s-1)!\big)^{2}}\sum_{j=1}^{d}\Bigg\lvert\dfrac{d^{s-1}E_{j}(0)}{dz^{s-1}}\Bigg\rvert^{2}.$$\end{proof}

    By Theorem \ref{normofblaschkevectors}, the norm of each vector in a Blaschke frame is a consequence of the Wold decomposition. The following test we state is to find the redundant and essential vectors of any frame, which was given by Christensen in \cite{Christensen}. This result is of the utmost importance. 

\begin{theorem}{\normalfont{\cite[Theorem 5.4.7]{Christensen}}}\label{redundancyresult}
Let $(x_{m})_{m=1}^{\infty}$ be a frame for a Hilbert space $\mathcal{H}$ with frame operator $\mathfrak{F}$. The removal of a vector $x_{k}$ from the frame $(x_{m})_{m=1}^{\infty}$ leaves either a frame or an incomplete set. More precisely:
    
    \begin{enumerate}
        \item If $\langle x_{k}, \mathfrak{F}^{-1}x_{k}\rangle_{\mathcal{H}}\neq 1$, then $(x_{m})_{m\neq k}^{\infty}$ is a frame for $\mathcal{H}$, that is, $x_{k}$ is redundant;
         \item  If $\langle x_{k}, \mathfrak{F}^{-1}x_{k}\rangle_{\mathcal{H}}=1$, then $(x_{m})_{m\neq k}^{\infty}$ is incomplete, that is, $x_{k}$ is essential.
    \end{enumerate}
    
\end{theorem}

\noindent The following result therefore follows from Theorem \ref{redundancyresult}.

\begin{corollary}\label{newcor}
        Let $\mathcal{H}$ be a Hilbert space and $B$ be a finite Blaschke product. Every vector $b_{k}$ in a Blaschke frame for $B$ with isometry $V$ on $\mathcal{H}$ for $k=1,2,\hdots$ such that $b_{k}\in\mathcal{H}_{1}$ is an essential vector, where $\mathcal{H}_{1}=\bigcap_{j=0}^{\infty}V^{j}\mathcal{H}$.
    \end{corollary}

    \begin{proof}
        As the frame operator for a Blaschke frame is the identity, Theorem \ref{blaschkeframetheorem}, then by Theorem \ref{redundancyresult}, the essential vectors of a Blaschke frame have norm one. By Theorem \ref{normofblaschkevectors}, if $b_{k}\in\mathcal{H}_{1}$, then $\|b_{k}\|_{\mathcal{H}}=1$ and thus essential.
    \end{proof}

In Theorem \ref{normofblaschkevectors}, if the isometry $V$ is a unitary unitary operator $U$ on a Hilbert space $\mathcal{H}$, then $1_{\mathcal{H}}-UU^{*}=0$ and 

$$\|b_{m}\|_{\mathcal{H}}^{2}=\Big\langle  B^{*}(U^{*})B(U^{*})e_{m},e_{m}\Big\rangle_{\mathcal{H}}=\bigg\langle  \Big(1_{\mathcal{H}}-\sum_{j=1}^{d}E^{*}_{j}(U^{*})(1_{\mathcal{H}}-UU^{*})E_{j}(U^{*})\Big)e_{m},e_{m}\bigg\rangle_{\mathcal{H}}=1.$$ 

\noindent Therefore, by the redundancy condition from Christensen, Theorem \ref{redundancyresult}, a Blaschke frame for $B$ with a unitary operator $U$ is an exact frame, or equivalently, a Riesz basis for a Hilbert space $\mathcal{H}$. In this instance, we have $(b_{m})_{m=1}^{\infty}=(B(U^{*})e_{m})_{m=1}^{\infty}$, therefore $B(U^{*})$ must be invertible. Indeed, by Lemma \ref{blaschkeoperatorcoisom}, $B(V^{*})$ is a co-isometry for any isometry $V$ on $\mathcal{H}$. By the model formula for the finite Blaschke product, Lemma \ref{blasckhelemma},

$$1_{\mathcal{H}}-B^{*}(U^{*})B(U^{*})=\sum_{j=1}^{d}E^{*}_{j}(U^{*})(1-UU^{*})E_{j}(U^{*})=0.$$

\noindent Thus $B^{*}(U^{*})B(U^{*})=1_{\mathcal{H}}$ and $B(U^{*})$ is also an isometry. Hence $B(U^{*})$ is invertible on $\mathcal{H}$ and $B^{-1}(U^{*})=B^{*}(U^{*})$, that is, $B(U^{*})$ is a unitary operator on $\mathcal{H}$. We have therefore proved the following.

\begin{proposition}\label{propunitary}
    A Blaschke frame for $B$ with a unitary operator $U$ on a Hilbert space $\mathcal{H}$ is a Riesz basis for $\mathcal{H}$. 
\end{proposition}

Proposition \ref{propunitary} coincides with Theorem \ref{normofblaschkevectors}, as the Wold decomposition for the unitary operator $U$ is trivial, that is, $\mathcal{H}=\mathcal{H}_{1}$ and $b_{m}\in\mathcal{H}_{1}$. We now explore alternative choices for a Blaschke frame with the next result.

\begin{lemma}\label{redundancycor}
Let $\mathcal{H}$ be a Hilbert space and let $(E_{j})_{j=1}^{d}$ be the Malmquist-Takenaka functions corresponding to a finite Blaschke product $B$. For every vector $b_{k}$ in a Blaschke frame for $B$ with isometry $V$ on $\mathcal{H}$ for $k=1,2,\hdots$ such that $b_{k}\not\in\mathcal{H}_{1}$, where $\mathcal{H}_{1}=\bigcap_{j=0}^{\infty}V^{j}\mathcal{H}$, there exists a positive integer $s$ such that:

\begin{enumerate}
     \item  If $\dfrac{d^{s-1}E_{j}(0)}{dz^{s-1}} = 0$ for $j=1,\hdots,d$, then $b_{k}$ is essential;
     \vspace{0.2cm}
         \item  If $\dfrac{d^{s-1}E_{j}(0)}{dz^{s-1}} \neq 0$ for at least one $j=1,\hdots,d$, then $b_{k}$ is redundant.
\end{enumerate}

\end{lemma}

\begin{proof}
To prove (1) and show incompleteness of the sequence $(b_{m})_{m\neq k}^{\infty}$, where $b_{m} = B(V^{*})e_{m}$ such that $b_{m}\not\in\mathcal{H}_{1}$, we need to find when

$$\langle b_{k} , \mathfrak{F}^{-1} b_{k} \rangle_{\mathcal{H}} =1 $$

\noindent where $\mathfrak{F}$ is the frame operator for a Blaschke frame, Theorem \ref{redundancyresult}. By Theorem \ref{blaschkeframetheorem}, $\mathfrak{F}=1_{\mathcal{H}}$ for a Blaschke frame for $B$ with isometry $V$. Therefore we need to calculate the norm of the vector $b_{k}$ to find the redundancy condition. By Theorem \ref{normofblaschkevectors}, for each $m=1,2,\hdots$, there exists a positive integer $s$ such that

$$\|b_{m}\|_{\mathcal{H}}^{2} = 1 - \dfrac{1}{\big((s-1)!\big)^{2}}\sum_{j=1}^{d}\Bigg\lvert\dfrac{d^{s-1}E_{j}(0)}{dz^{s-1}}\Bigg\rvert^{2}.$$

\noindent Therefore, if $\|b_{k}\|_{\mathcal{H}}=1$, the sequence $(b_{m})_{m\neq k}^{\infty}$ is incomplete in $\mathcal{H}$. Under this assumption,

\begin{align*}
    \|b_{k}\|^{2}_{\mathcal{H}}=1 &\iff \dfrac{1}{\big((s-1)!\big)^{2}}\sum_{j=1}^{d}\Bigg\lvert\dfrac{d^{s-1}E_{j}(0)}{dz^{s-1}}\Bigg\rvert^{2}=0\\
    &\iff \Bigg\lvert\dfrac{d^{s-1}E_{j}(0)}{dz^{s-1}}\Bigg\rvert^{2} = 0~~\text{for}~j=1,\hdots,d\\
    &\iff \dfrac{d^{s-1}E_{j}(0)}{dz^{s-1}}=0~~\text{for}~j=1,\hdots,d.
\end{align*}

To prove (2), if $\dfrac{d^{s-1}E_{j}(0)}{dz^{s-1}}\neq 0$ for at least one $j=1,\hdots,d$

$$\|b_{k}\|_{\mathcal{H}}^{2} = 1 - \dfrac{1}{\big((s-1)!\big)^{2}}\sum_{j=1}^{d}\Bigg\lvert\dfrac{d^{s-1}E_{j}(0)}{dz^{s-1}}\Bigg\rvert^{2}\neq 1.$$

\noindent By Theorem \ref{redundancyresult}, the vector $b_{k}$ is redundant and $(b_{m})_{m\neq k}$ is a frame for $\mathcal{H}$.\end{proof}

For a pure isometry $V$ on a Hilbert space $\mathcal{H}$, we consider the following simple example to illustrate the test for redundancy.

\begin{example}\label{simpleexample}
Let $b_{m}$ be a vector in a Blaschke frame for $z^{d}$ with the right shift operator $S$ on a Hilbert space $\mathcal{H}$. Then every vector $b_{m}$ for $m\leq d$ is a redundant vector, that is, the sequence $(b_{k})_{k\neq m}$ for $m\geq d+1$ is complete.
    
\end{example}

\begin{proof}
For the right shift operator $S$ on $\mathcal{H}$, we have 

$$(b_{m})_{m=1}^{\infty}=((S^{*})^{d}e_{m})_{m=1}^{\infty}=(0)_{m=1}^{d}\cup(e_{m})_{m=1}^{\infty}.$$

\noindent As $z^{d}$ has one zero, $z=0$, of multiplicity $d$, the Malmquist-Takenaka functions $E_{j}$ corresponding to $z^{d}$ are given by $E_{j}(z)=z^{j-1}$ for all $z\in{\overline{\mathbb{D}}}$ and $j=1,\hdots,d$. By Lemma \ref{redundancycor}, for every positive integer $s$ such that $s\geq d+1$, 

\begin{equation}\label{shiftexampleder}
\dfrac{d^{s-1}E_{j}(0)}{dz^{s-1}}=0
\end{equation}

\noindent for $j=1,\hdots,d$. Thus every vector $b_{m}=(S^{*})^{d}e_{m}$ such that $m\geq d+1$ is essential, that is, $(b_{k})_{k\neq m}$ is incomplete. Moreover, every vector $b_{m}$ such that $m\leq d$ is redundant. This makes sense, as the frame sequence $(b_{m})_{m=1}^{\infty}$ is simply the original orthonormal basis union with $d$ copies of the zero vector. Therefore we can delete each zero vector at no cost to the frame, that is, we can remove the first $d$ vectors of the Blaschke frame for $z^{d}$ with the right shift operator $S$ on $\mathcal{H}$ such that $(b_{k})_{k\neq m}^{\infty}$ for $m\geq d+1$ is a frame for $\mathcal{H}$.\end{proof}

 Example \ref{simpleexample} extends beyond just the shift operator. Independent of the Hilbert space $\mathcal{H}$ and the isometry $V$ on $\mathcal{H}$, if $B(z)=z^{d}$ for all $z\in\overline{\mathbb{D}}$, the functions $E_{j}$ always satisfy equation (\ref{shiftexampleder}) for all positive integers $s\geq d+1$ and $j=1,\hdots,d$. Moreover, for all positive integers $s\leq d$,

 $$\dfrac{d^{s-1}E_{j}(0)}{dz^{s-1}}\neq 0$$

 \noindent for at least one $j=1,\hdots,d$, as
 
 $$\dfrac{d^{j-1}}{dz^{j-1}}E_{j}(0)=(j-1)!$$
 
\noindent for $j=1,\hdots,d$. Therefore we have proved the following result.

\begin{proposition}\label{redundantforzd}
Let $\mathcal{H}$ be a Hilbert space. The vector $b_{m}=(V^{*})^{d}e_{m}$ in a Blaschke frame for $z^{d}$ with isometry $V$ on $\mathcal{H}$ is redundant if and only if $e_{m}=S^{s-1}e_{\Lambda(n)}$ and $s\leq d$, where $S=V|_{\mathcal{H}_{2}}$ and $e_{\Lambda(n)}\in\mathrm{Ker}(V^{*})$ such that $\Lambda(n)\neq 0$.
\end{proposition}

% The next example demonstrates that redundant vectors are not only the zero vector.

% \begin{example}
%     Let $b_{m}$ be a vector in a Blaschke frame for $z^{d}$ with the unilateral shift $M_{p}$ on $L^{2}(0,1)$. For $s-1\leq d$, every vector $b_{m}=e_{np^{s-1}}$ is a redundant vector, where $n\not\equiv_{p}0$.
% \end{example}

% \begin{proof}

%     For the unilateral shift $M_{p}$ on $L^{2}(0,1)$, we have, by equation (\ref{mpproperty}),

% $$(b_{m})_{m=1}^{\infty}=((M_{p}^{*})^{d}e_{m})_{m=1}^{\infty}=(M_{p^{-d}}e_{np^{s-1}})_{m=1}^{\infty}=(M_{p^{-d}}e_{np^{s-1}})_{m=1}^{\infty}$$

% \noindent where $m=np^{s-1}$ for $n\not\equiv_{p}0$.

% \noindent As $B$ has one zero, $z=0$, of multiplicity $d$, the Malmquist-Takenaka functions $E_{j}$ corresponding to $B$ for $j=1,\hdots,d$ are given by $E_{j}(z)=z^{j-1}$ for all $z\in{\overline{\mathbb{D}}}$. By Corollary \ref{redundancycor}, for a positive integer $s$ such that, for all $s-1\geq d+1$, 

% \begin{equation}\label{shiftexampleder}
% \dfrac{d^{s-1}E_{j}(0)}{dz^{s-1}}=0
% \end{equation}

% \noindent for $j=1,\hdots,d$. Thus, for every vector $b_{m}=(S^{*})^{d}e_{m}$ such that $m\geq d+1$, $(b_{k})_{k\neq m}$ is incomplete, that is, every vector $b_{m}$ such that $m\leq d$ is redundant. This makes sense, as the frame sequence $(b_{m})_{m=1}^{\infty}$ is simply the original orthonormal basis union with $d$ copies of the zero vector. Therefore we can delete each zero vector at no cost to the frame.\end{proof}

Observe an implication of Lemma \ref{redundancycor}. By Proposition \ref{propunitary}, a Blaschke frame for $B$ with a unitary operator $U$ on $\mathcal{H}$ is a Riesz basis. Therefore every vector in a Blaschke frame is essential. Instead, suppose $V$ is a pure isometry. If $B=z^{d}$, then Proposition \ref{redundantforzd} tells us the redundancy condition for this Blaschke frame. The only other alternative is found in our final result, which says that vectors in a Blaschke frame for a finite Blaschke product with at least one non-zero root respect the Wold decomposition.

% let $\mathcal{H}=\mathcal{H}_{1}\oplus\mathcal{H}_{2}$ be a Hilbert space and let $b_{m}$ be a vector in a Blaschke frame for $B$ with isometry $V$ on $\mathcal{H}$. If the Maclaurin coefficients for at least one of the Malmquist-Takenaka functions $(E_{j})_{j=1}^{d}$ corresponding to $B$ are all non-zero, then every vector $b_{m}$ is redundant, that is, we can delete any vector $b_{m}$ in a Blaschke frame such that $(b_{k})_{k\neq m}$ is a frame. 

\begin{theorem}\label{everyvecisred}
    Let $B$ be a finite Blaschke product with at least one non-zero root and let $(b_{k})_{k=1}^{\infty}$ be a Blaschke frame for $B$ with a pure isometry $V$ on a Hilbert space $\mathcal{H}$:

    \begin{enumerate}
        \item if $b_{k}\in\mathcal{H}_{1}=\bigcap_{j=0}^{\infty}V^{j}\mathcal{H}$, then $b_{k}$ is essential;
        \item if $b_{k}\in\mathcal{H}_{2}=\bigoplus_{j=0}^{\infty}V^{j}\mathcal{L}$ for $\mathcal{L}=\mathrm{Ker}(V^{*})$, then $b_{k}$ is redundant.
    \end{enumerate}

\end{theorem}

\begin{proof}
Statement (1) follows directly from Corollary \ref{newcor}. To prove statement (2), if $B$ a finite Blaschke product of degree $d$ with at least one non-zero root, say $\lambda\in\mathbb{D}$, we can write

    $$B(z) = z^{k} \dfrac{z+\lambda}{1+\overline{\lambda}z}\widetilde{B}(z)$$
    
\noindent where $\widetilde{B}$ is a finite Blaschke product of degree $d-k-1$ with no roots at zero. The first $k+1$ Malmquist-Takenaka functions corresponding to $B$ satisfy $E_{j}(z)=z^{j-1}$ for $j=1,\hdots,k$ and 

$$E_{k+1}(z)=\dfrac{z^{k}(1-\lvert \lambda\rvert^{2})^{1/2}}{1+\overline{\lambda}z}.$$

\noindent By Lemma \ref{redundancycor}, if there exists a positive integer $s$ such that

$$\dfrac{d^{s-1}E_{j}(0)}{dz^{s-1}} = 0$$

\noindent for $j=1,\hdots,d$, then the vector $b_{m}=B(V^{*})e_{m}$ such that $S^{s-1}e_{\Lambda(n)}=e_{m}$ and $e_{\Lambda(n)}\in\mathrm{Ker}(V^{*})$ is essential for $m=1,2,\hdots$. Note

$$\dfrac{d^{k}E_{j}(0)}{dz^{k}}=0$$ 

\noindent for $j=1,\hdots,k$ but, for the function $E_{k+1}$,

$$E_{k+1}(z)=(1-\lvert \lambda\rvert^{2})^{1/2}\big(z^{k}-\overline{\lambda}z^{k+1}+\overline{\lambda}^{2}z^{k+2}-\overline{\lambda}^{3}z^{k+3}+\hdots\big)$$

\noindent and

$$\dfrac{d^{s-1}E_{k+1}(0)}{dz^{s-1}}\neq 0$$

\noindent for $s\geq k+1$. Hence, there does not exist a positive integer $s$ such that all the Malmquist-Takenaka functions satisfy Lemma \ref{redundancycor}, (1). Instead, for every positive integer $s$, Lemma \ref{redundancycor}, (2) is satisfied. Thus, for a finite Blaschke product $B$ with at least one non-zero root, a vector $b_{k}\in\mathcal{H}_{2}$ of a Blaschke frame for $B$ with a pure isometry $V$ is redundant. \end{proof}

% Another example is when a Blaschke frame is an exact frame, or equivalently, a Riesz basis.

% \begin{example}
%     Let 
    
%     $$B(z)=\dfrac{z^{2}-\frac{1}{4}}{1-\frac{1}{4}z^{2}}$$ 
    
% \noindent for all $z\in\overline{\mathbb{D}}$ and let $b_{m}$ be a vector in a Blaschke frame for $B$ with the unitary operator $U$ on $l^{2}(\mathbb{N})$ defined by 
    
%     $$U=\mathrm{Diag}\Big[e^{\frac{1}{2}i\theta_{m}}\Big]_{m=1}^{\infty}$$
    
% \noindent for $\theta_{m}\in[0,2\pi)$. Then every vector $b_{m}$ for $m\geq 1$ is a redundant vector, that is, the sequence $(b_{m})_{m=1}^{\infty}$ is a Riesz basis.
% \end{example}

% \begin{proof}
%     By Proposition \ref{propunitary}, $B(U^{*})$ is a unitary operator on $l^{2}(\mathbb{N})$. Therefore 

% $$B(U^{*})=\mathrm{Diag}\bigg[\dfrac{e^{-i\theta_{m}}-\frac{1}{4}}{1-\frac{1}{4}e^{-i\theta_{m}}}\bigg]_{m=1}^{\infty}.$$

% \noindent For the standard orthonormal basis $(e_{m})_{m}^{\infty}$ of $l^{2}(\mathbb{N})$ defined by

% $$ e_{m}=( \delta_{1,m},\delta_{2,m},\hdots,)$$

% \noindent where $\delta_{j,m}$ is the Kronecker delta function for $j=1,2,\hdots$, we have

% $$(b_{m})_{m=1}^{\infty} =\bigg( \dfrac{e^{-i\theta_{m}}-\frac{1}{4}}{1-\frac{1}{4}e^{-i\theta_{m}}}e_{m}\bigg)_{m=1}^{\infty}$$

% \noindent is a Riesz basis of $l^{2}(\mathbb{N})$ for $\theta_{m}\in[0,2\pi)$.\end{proof}

{\bf Acknowledgements.} I would like to thank Lyonell Boulton, Zinaida Lykova and Nicholas Young, for their indispensable comments on earlier iterations of this work. I would also like to thank the referees for their kind and constructive remarks.

\bibliographystyle{amsplain}

\end{document}